\documentclass[10pt]{article}
\usepackage[all,2cell]{xy} \UseAllTwocells \SilentMatrices
\usepackage{latexsym,amsfonts,amssymb}
\usepackage{amsmath,amsthm,amscd}
\usepackage{hyperref,psfrag}
\usepackage{diagcat}
\usepackage{color}
\usepackage{etoolbox}
\usepackage[utf8]{inputenc}
\usepackage{graphicx}
\usepackage[dvips]{epsfig}
\usepackage[centertableaux]{ytableau}
\usepackage{a4wide}
\usepackage{geometry}

\usepackage{tikz}
\usepackage{overpic}

\usepackage{epigraph,wrapfig}

\newstrandstyle{Red}{type=ribbon,style=solid,ribboncolor=lightred,color=red}%%
\normalfont\upshape

\usepackage{fancyhdr}
\usepackage[all]{xy}
\SelectTips{cm}{}

\theoremstyle{plain}
\newtheorem{thm}{Theorem}[section]
\newtheorem{prop}[thm]{Proposition}
\newtheorem{lemma}[thm]{Lemma}
\newtheorem{cor}[thm]{Corollary}

\theoremstyle{definition}

\newtheorem{defn}[thm]{Definition}
\newtheorem{eg}[thm]{Example}
\newtheorem{rmk}[thm]{Remark}
\newtheorem{ntn}[thm]{Notation}

\numberwithin{equation}{section}

\usepackage{bbm}
\def\1{\mathbbm 1}

\def\Z{\mathbb Z}
\def\N{\mathbb N}
\def\C{\mathbb C}

\def\o{\otimes}
\def\lra{\longrightarrow}

\def\RHOM{\mathbf{R}\mathrm{HOM}}
\def\Id{\mathrm{Id}}
\def\mc{\mathcal}
\def\mf{\mathfrak}
\def\Ext{\mathrm{Ext}}
\def\End{\mathrm{End}}
\def\END{\mathrm{END}}

\newcommand{\dmod}{\!-\!\mathrm{mod}}

\newcommand{\dif}{d}   %% differential

\newcommand{\mH}{\mathrm{H}}  %% homology
\newcommand{\mHH}{\mathrm{HH}}  %% homology
\newcommand{\Hom}{{\rm Hom}}
\newcommand{\HOM}{{\rm HOM}}

\def\dif{d}

\def\lra{{\longrightarrow}}
\def\dmod{{\mathrm{-mod}}}   %% finitely-generated modules
\def\Ext{{\mathrm{Ext}}}

\def\Id{\mathrm{Id}}
\def\mc{\mathcal}
\def\mf{\mathfrak}
\def\shuffle{\,\raise 1pt\hbox{$\scriptscriptstyle\cup{\mskip
               -4mu}\cup$}\,}
\newcommand{\refequal}[1]{\xy {\ar@{=}^{#1}
(-1,0)*{};(1,0)*{}};
\endxy}

\usepackage[vcentermath]{youngtab}

\newstrandstyle{object}{type=ribbon,style=solid,ribboncolor=gray,ribbonbdrycolor=black,width=18pt}
\newstrandstyle{object2}{type=ribbon,style=solid,ribboncolor=blue,ribbonbdrycolor=black,width=18pt}

\title{A braid group action on an $A_\infty$-category for zigzag algebras}
\author{Benjamin Cooper, You Qi, Joshua Sussan}

\date{May 4, 2023}
\begin{document}
%
% ==============================================================================

\maketitle

\begin{center}
  \emph{Dedicated to Igor Frenkel on the occasion of his seventieth birthday}
\end{center}

\setcounter{tocdepth}{1}

\begin{abstract}
%We construct a DG structure on a zigzag algebra used by Khovanov and Seidel to construct a categorical braid group %action.  We show there is a braid group action in this DG setting.
We construct differential graded enhancements of the zigzag algebras which were used by Khovanov, Seidel and Thomas to produce categorical braid group actions. These enhancements are related to $p$-differential graded structures by a version of Koszul duality. We prove that the minimal model $A_\infty$-structure on the zigzag algebras is {\em not} formal. We construct a braid group action in this setting and suggest a symplectic interpretation.
\end{abstract}

\tableofcontents

\section{Introduction}
In the 1990s, Crane and Frenkel conjectured that the 3-dimensional topological field theories of Witten, Reshetikhin and Turaev could be rewritten in the new categorical language that was appearing in geometric representation theory at the time \cite{CF}. Over the years,  these ideas have revealed a vast mathematical landscape. Although the conjecture remains open, many parts of what we know today are expected to play a role in the final answer. Among these parts, the zigzag algebras $C_{n-1}$, where $n \geq 2$ is a natural number, introduced by Khovanov and Seidel \cite {KS}, also \cite{RouZim}, serve as a kind-of simplest non-trivial example: sophisticated enough for its bounded homotopy category $K^b(C_{n-1})$ to support a faithful categorical representation of the braid group 
$$\Psi_n : \mathrm{Br}_n\to \mathrm{Aut}(K^b(C_{n-1}))\quad\quad\quad\quad \mathrm{Ker}(\Psi_n)=\{e\},$$
yet simple enough to appear in many places. There are interpretations of $C_{n-1}$ in terms of the category $\mathcal{O}(\mathfrak{sl}_{n})$ and perverse sheaves on $\mathbb{P}^{n-1}$. The zigzag algebra also arises naturally in homological realization of Nakajima quiver varieties \cite{FKO}, as well as the endomorphism algebra of a system of Lagrangian spheres \cite{KS}. This latter observation was exploited by Seidel and Thomas to construct categorical braid group actions on many different examples \cite{SeidelThomas}.
% The zigzag algebra has found a role to play in many areas of mathematics.
%This zig-zag piece of Crane and Frenkel's dream has found a role to play in the jingle-jangle morning that is present-day mathematics.

In a slightly different direction, Chen, Khovanov, and independently Stroppel, studied, as a special case of their constructions, the Grothendieck group $K_0(A_n)\otimes \mathbb{C}$ of graded modules over a certain algebra $A_n$ related to $C_{n-1}$, proving that it can be identified with the second highest weight space in the representation $(\mathbb{C}(q)^2)^{\otimes n}$ of $U_q(\mathfrak{sl}_2)$ \cite{ChenKho, StroppelGrass}. Under this correspondence, the homomorphism $\Psi_n$ becomes the Burau representation, well-studied in low-dimensional topology \cite{birman}. 

In \cite{Hopforoots}, Khovanov suggested to categorify the operation of setting $q$ to be a $p$th root of unity by introducing a $p$-nilpotent derivation
$$\partial : A \to A  \quad\quad\textnormal{ such that }\quad\quad \partial^p=0$$
and passing from the category of graded $A$-modules to a triangulated category of $(A,\partial)$-modules, (see Section \ref{sec:prelim}). This led Khovanov and the second author to find $p$-nilpotent operators on nilHecke algebras, Khovanov-Lauda-Rouquier algebras and Webster algebras underlying many categorifications of quantum groups and their representations \cite{KQ, EQ1, KQS, QiSussan3}. The observation that the Koszul dual $A_n^!$ of the algebra $A_n$ admits such a $p$-derivation $\partial \in \End(A_n^!)$ compatible with its identification as the Webster algebra associated to the second highest weight of the $U_q(\mathfrak{sl}_2)$ representation $(\mathbb{C}(q)^2)^{\otimes n}$ led to a detailed study of $(A_n^!, \partial)$, ultimately constructing an analogue of the categorical representation $\Psi_n$ above evaluated at a $p$th root of unity \cite{QiSussan1, QiSussan2}.

This paper continues the study of zigzag algebras by examining what happens to this $p$-differential graded structure $\partial : A_n^! \to A_n^!$ under Koszul duality. In short, 
\begin{center} what is $(A_n^!,\partial)^!$? \end{center}
We answer this question when $p=2$. In this case, $(A^!_n, \partial)$ becomes a differential graded algebra and we compute its Koszul dual $A_\infty$-algebra. This is done in stages.
There is a natural bijection between the simple modules $\{L_i\}$ of $A_n^!$ and those of $(A_n^!,\partial)$. We observe that the cofibrant resolutions $\mathbb{P}(L_i)\to L_i$ in Lemma \ref{lemma-DG-ny-resolution} of the simple modules $L_i$ over $(A_n^!, \partial)$ are closely related to, but not the same as projective resolutions of corresponding simple modules $L_i$ over $A_n^!$. Then we construct chain maps between resolutions associated to the generators of the algebra $A_n$. The relations that these chain maps satisfy allow us to introduce differential graded algebras $(E_n, d)$ which act as intermediaries between $(A^!_n,\partial)$ and $A_n$. More precisely,
\begin{enumerate}
    \item There are canonical inclusions of differential graded algebras: $E_n \hookrightarrow \END(\oplus_i \mathbb{P}(L_i))$ inducing isomorphisms $\mH(E_n,d) \cong \Ext^*(\oplus_i L_i, \oplus_i L_i)$.
    \item There are canonical isomorphisms of algebras: $\mH(E_n, d) \cong A_n$.
\end{enumerate}
In this sense, the algebra $(E_n,d)$ is a differential graded (DG) enhancement of the algebra $A_n$. A setting such as this is ripe for an application of the transfer theorem \cite{KeLef, CoSussan}: the algebra $A_n \cong \mH(E_n,d)$ supports a minimal $A_\infty$-structure $(A_n, m)$ where $m = \{ m_r : A_n^{\otimes r}\to A_n[2-r] \}_{r\geq 1}$ with $m_1=0$ which makes $(A_n,m)$ equivalent to $(E_n,d)$ as $A_\infty$-algebras. The question posed above can be answered in equations as
$$(A_n^!,\partial)^! \cong (E_n, d) \cong (A_n, m).$$
Section \ref{sec:ainfzz} contains a computation of this minimal $A_\infty$-algebra structure for 
%the stabilization $\mH(E_\infty,d)$ of 
the algebra $(E_n,d)$. This stable minimal $A_\infty$-algebra is non-trivial, see Theorem \ref{thm:ainfformulas} for formulas, which shows that the $A_\infty$-structures on unstable algebras $A_n$ and $C_{n-1}$ must be non-trivial as well. This result is a little bit surprising because Seidel and Thomas used formality, the vanishing of all $A_\infty$-structures, on zigzag algebras to build their braid group representations see \cite[Lemma 4.21]{SeidelThomas}. Our $A_\infty$-structure on zigzag algebras does not contradict Seidel and Thomas's theorem since a different grading is used, more precisely the operations $m_3 \in CC^3(C_n,C_n)^{-1}$ 
(see Section \ref{sec:ainfzz})
live in a degree of Hochschild cohomology which is not covered by their vanishing result.
 Corresponding vanishing statements feature in related works \cite{LiuWang, BarWang}.

In Section \ref{sec:braidaction}, we check that the derived category of the $A_\infty$-algebra zigzag algebra supports a braid group action. This is done by transporting the braid group action defined in the context of $(A_n^!, \partial)$ using the derived Morita equivalence obtained from the Koszul duality construction. The fact that the $A_\infty$-structure $m$ on the zigzag algebra facilitates a braid group representation, is strong evidence to suggest its relevance and naturality within the context of the Crane and Frenkel conjectures.

Finally, zigzag algebras appear in a variety of contexts. It is natural to ask whether one can interpret this new $A_\infty$-structure within one of those contexts. In Section \ref{sec:symp}, we recall how $A_n$ or $C_{n-1}$ can be found in the Fukaya category of a surface, why there is a braid group representation and see that the equations we discovered from the $A_\infty$ perturbation lemma can be understood in terms of relationships among disks and curves on the surfaces.

\paragraph{Organization.}
%This paper is organized as follows. 
We review some tools from homological algebra in Section \ref{sec:prelim}. 
%and $p$-homological algebra, 
In Section \ref{sec:dualzigzagalgebra}, the dual zigzag algebra is introduced, resolutions $\mathbb{P}(L_i)$ of simples $L_i$ are constructed as well as chain maps between these resolutions. %corresponding to generators of $A_n$. 
In Section \ref{sec:ainfzz} the DG algebra $(E_n,d)$ and its subalgebra $(S_n,d)$ are introduced and the transfer theorem is used to compute the $A_\infty$-structure on the cohomology of $S_n$, which is the zigzag algebra $C_{n-1}$. In Section \ref{sec:braidaction} the categorical braid group representation is constructed. In Section \ref{sec:symp}, the stable $A_\infty$-structure is examined using curves on a surface.

\paragraph{Notations.}
The ground ring will be a field of characteristic 2. Below is a list of the algebras that will play a role in this paper:
\begin{enumerate}
\item $A_n^!$ is the Koszul algebra (one of the simples is not spherical) which has the original DG structure.
\item $A_n$ is the Koszul dual of $A_n^!$ (one of the projectives is not spherical).  This is almost the zigzag algebra.
\item $C_{n-1}$ is a subalgebra of $A_n$.  This is the zigzag algebra and all of the $n-1$ projectives are spherical.
\item $E_n$ is the DG endomorphism algebra of simples of $A_n^!$, and its homology is $A_n$.
\item $E_n^\prime$ is a certain subalgebra of $E_n$ whose homology is $A_n$.
\item $S_{n-1}$ is a DG subalgebra of $E_n$.  Its homology is $C_{n-1}$ (the zigzag algebra).
\end{enumerate}
Most of these algebras are defined at the beginning of Section \ref{sec:ainfzz}. The diagram below may be useful.
$$\begin{tikzpicture}[scale=10, node distance=2cm]
\node (A) {$(S_{n-1},d)$};
\node (C) [right of=A] {$(E_n,d)$};

\node (B) [below of=A] {$C_{n-1}$};

\node (D) [below of=C] {$A_n$};
\draw [->] (A) -- (C) node[midway,above]  {$\subset$};
\draw [->] (B) -- (D) node[midway,above]  {$\subset$};
\draw[|->] (A) -- (B) node [swap,midway,left] {$\mH$};
\draw[|->] (C) -- (D) node [swap,midway,right] {$\mH$};
\end{tikzpicture}$$

\paragraph{Acknowledgements.} 
The authors would like to thank Igor Frenkel for always sharing with us his panoramic view of mathematics and physics, as well as his constant encouragement.

J.S. would like to thank Ben Elias and Anthony Licata for discussions related to the DG enhanced zigzag algebra.

Y.Q. is partially supported by a Simons Foundation Collaboration Grant for Mathematicians.

B.C. thanks K. Kawamuro and M. Tehrani for their cordial collegiality. His participation was facilitated by Simons Award \#638089.

%%%%%%%%%%%%%%%%%%%%%%%%%%%%%%%%%%
%%%%%%%%%%%%%%%%%%%%%%%%%%%%%%%%%%
%%
%%      Preliminaries
%%
%%%%%%%%%%%%%%%%%%%%%%%%%%%%%%%%%%
%%%%%%%%%%%%%%%%%%%%%%%%%%%%%%%%%%

\section{Preliminaries} \label{sec:prelim}

\paragraph{The braid group.}
Let $q$ be a formal variable, and $\Z[q,q^{-1}] $ be the ring of Laurent polynomials. The Temperley-Lieb algebra $\mathrm{TL}_n$ over $\Z[q,q^{-1}]$ is generated by elements $ u_i $ for $i =1, \ldots, n-1 $ subject to the relations that
\begin{enumerate}
\item[(i)] $u_i u_j= u_j u_i$ if $|i-j|>1$,
\item[(ii)] $ u_i^2=(1+q)u_i $ for $ i=1,\ldots, n-1$,
\item[(iii)] $ u_i u_{i + 1} u_i = q u_i $ for $ i=1,\ldots, n-1$ and $ u_i u_{i - 1} u_i = q u_i $ for $ i=2,\ldots, n$.
\end{enumerate}
Consider the  $\Z[q,q^{-1}]$-lattice $V_{n-1}=\oplus_{i=1}^{n-1}\Z[q,q^{-1}] v_i$. Define
\begin{equation}
    u_i(v_j):=
    \begin{cases}
    (1+q)v_i & i=j\\
    v_i & j=i-1\\
    qv_i & j=i+1\\
    0 & |i-j|>1
    \end{cases}
\end{equation}
Base changing from $\Z[q,q^{-1}]$ to $\C$ by sending $q$ to a chosen complex number, we call the resulting $(n-1)$-dimensional representation the (reduced) \emph{Burau representation}, which is still denoted $\mathrm{V}_{n-1}$. 

The braid group $\mathrm{Br}_n$ is generated by elements $ t_i $ for $i =1, \ldots, n-1 $ subject to the relations that
\begin{enumerate}
\item[(i)] $t_i t_j= t_j t_i$ if $|i-j|>1$,
\item[(ii)] $ t_i t_{i +1} t_i = t_{i+1} t_i t_{i+1} $ for $ i=1,\ldots,n-2$.
\end{enumerate}
For $i=1,\ldots, n-1$, sending the braid group generators $t_i$ to $1-u_i$ defines a homomorphism from the group algebra of $\mathrm{Br}_n$ into $\mathrm{TL}_n$. Via this homomorphism, we can also regard $V_{n-1}$ as a representation of the braid group.

\paragraph{Homological algebra.} We gather in this section some necessary homological algebra background that we shall use. We refer the reader to the references \cite{Ke1, KeICM} for more details. 

Fix once and for all a base field $\Bbbk$ of characteristic $2$. Unadorned tensor product $\o$ will stand for the tensor product over $\Bbbk$. Throughout we use the French convention $\N:=\{0,1,2,\dots\}$. 

A \emph{differential graded (DG) algebra} $(A,\dif_A)$ over $\Bbbk$ consists of a $\Z$-graded algebra $A=\oplus_{i\in \Z}A^i$ and a degree one endomorphism $\dif_A$ of $A$ that is $2$-nilpotent and satisfies the Leibniz rule:
\[
\dif_A^2\equiv 0,  \ \ \ \ \dif_A(ab)=\dif_A(a)b+a\dif_A(b),
\]
for any $a,b\in A$. A left (resp. right) DG module $(M,\dif_M)$ over $(A,\dif_A)$ is a left (resp. right) $A$-module equipped with a degree one, $2$-nilpotent endomorphism $\dif_M$, $\dif_M^2\equiv 0$, which is compatible with the differential action on $A$:
\[
\dif_M(ax)=\dif_A(a)x+a\dif_M(x) \ \ \ \ (\textrm{resp.}~\dif_M(xa)=\dif_M(x)a+x\dif_A(a)).
\]
For a DG module $ (M,\dif_M) $ we will also consider the shifted module $ (M[i],\dif_{M[i]}) $ where the degree $j$ component $ M[i]^j$ is equal to $M^{i+j}$.
In what follows, we will be exclusively considering DG algebras of the following type.
\begin{defn}\label{defn-positive-p-DGA}A DG algebra $A$ is called \emph{positive}, if it satisfies the following three conditions.
\begin{itemize}
\item[(i)] The algebra $A$ is non-negatively graded, $A\cong \oplus_{i\in \N}A^i$, and is finite-dimensional in each fixed degree.
\item[(ii)] The homogeneous degree-zero part $A^0$ is semi-simple.
\item[(iii)] The differential $\dif_A$ acts trivially on $A^0$.
\end{itemize}
\end{defn}

The subscripts in the differentials on the DG algebras and modules will be dropped when no confusion can be caused. We adapt some basic definitions from \cite[Section 2]{KQ} in the situation when $A$ is positive, and the reader is referred loc.~cit.~for a summary of the basic definitions and techniques of the homological theory of DG modules. Some further information on the subject in the context of \emph{hopfological algebra} can be found in \cite{QYHopf}. Let $H=\Bbbk[d]/(d^2)$.  We will use the notations $A_\dif\dmod$ (resp.~$\mc{C}(A)$, $\mc{D}(A)$) to stand for the abelian (resp.~homotopy, derived) category of DG modules over $A$. Here $A_\dif$ denotes the \emph{smash product algebra}\footnote{It is usually denoted by $A \# H$.}, which is isomorphic, as a vector space, to $A\o H$, subject to the relations that $A\cong A\o 1\subset A\o H$, $H\cong  1\o H \subset A\o H$ as subalgebras, and the rule for commuting elements
\[
(1\o \dif) \cdot (a\o 1 )=(a\o 1) \cdot(1\o \dif) + \dif(a)\o 1,
\]
where $a$ is an arbitrary element of $A$. 

When $A$ is a positive DG algebra, $A_d$ is also a positively graded algebra with degree zero part isomorphic to $A^0$. Choose a complete set of pairwise non-isomorphic indecomposable idempotents $\{e_i|i\in I\}$ in $A^0$, and set $P_i:=Ae_i$. By construction, $P_i$'s are DG summands of the rank-one free module $A$.

As a matter of convention, given two graded modules $M$, $N$ over a graded algebra $A$, we set $\Hom_{A}(M,N)$ to be the space of degree preserving $A$-module maps between $M$ and $N$. We define
$$\HOM_{A}(M,N):=\oplus_{l\in \Z}\Hom_{A}(M,N[l]) .$$
%where $N[l]$ stands for the same underlying $A$-module $N$ with gradings shifted down by $l$: the homogeneous degree $i$ part of $N[l]$ is defined to be the homogeneous degree $i+l$ part of $N$.

The categories $\mc{C}(A)$, $\mc{D}(A)$ are triangulated, equipped with the invertible \emph{shift} endo-functors $[l]$ for any $l\in \Z$. 

\begin{defn}\label{def-nice-p-dg-module}Let $A$ be a DG algebra, and $M$ be a DG module over $A$.
\begin{itemize}
\item[(i)] The module $M$ is called \emph{cofibrant} if for any surjective quasi-isomorphism of DG modules $N_1\twoheadrightarrow N_2$, the induced map
\[
\HOM_A(M,N_1)\lra \HOM_A(M,N_2)
\]
of complexes is a homotopy equivalence.
\item[(ii)] When $A$ is positive, we say $M$ is a \emph{finite cell module} if it has a finite exhaustive increasing filtration $F^\bullet$ by DG submodules such that the subquotients
\[
F^\bullet/F^{\bullet-1}\cong P_i[l_i]
\]
for some projective $A$-modules $P_i$, $i=1,\dots ,n$ and $l_i\in \Z$.
\end{itemize}
\end{defn}

Cofibrant modules are ``nice'' in the sense that we have a good control of the morphism spaces between them in the derived category $\mc{D}(A)$. Namely, there is, for any cofibrant module $M$, a functorial-in-$N$ isomorphism of $\Bbbk$-vector spaces
\[
\Hom_{\mc{C}(A)}(M,N)\cong \Hom_{\mc{D}(A)}(M,N),
\]
which is induced from the localization functor $\mc{C}(A)\lra \mc{D}(A)$. It is an easy exercise to show that finite cell modules are cofibrant. It is clear that if $V$ is a finite-dimensional complex $V$, then $A\o V$ is a finite cell module. Finite cell modules in the DG setting are also known as ``one-sided twisted complexes'' of projectives \cite{BondalKapranov}.

\begin{eg}\label{eg-finite-cofibrant-mod}Let $A$ be a DG algebra. Assume that $a_0$ is an element of $A$ with $\dif_A(a_0)=0$. Here we construct an example of a finite cell module that is not of the form $A\o V$.

Consider a free $A$-module on two generators $M:=Av_0\oplus Av_1$, where $\mathrm{deg}(v_0)-\mathrm{deg}(v_1)=1-\mathrm{deg}(a_0)$. Define a differential $\dif_M$ as follows. On generators, it acts by
\[
\dif_M(v_0)=0, \ \ \ \ \dif_M(v_1)=a_0v_0.
\]
It is extended to all of $M$ by the Leibniz rule. We leave it as an exercise for the reader to check that $\dif_M^2\equiv 0$, and $M$ is not isomorphic to the tensor product of $A$ by a two-dimensional complex.
\end{eg}

The morphism space $\Hom_{\mc{D}(A)}(M,N)$ between two DG modules $M,N\in \mc{D}(A)$ is canonically isomorphic to the degree zero cohomology of the complex $\RHOM_{A}(M,N)$
\begin{equation}\label{eqn-morphism-space-as-invariants}
\Hom_{\mc{D}(A)}(M,N) \cong \mH^0( \RHOM_{A}(M,N))
\end{equation}
where $\RHOM_A(M,N):= \HOM_A(\mathbb{P}(M),N)$ is computed using a cofibrant replacement $\mathbb{P}(M)\to M$ of $M$.
More generally, attached to any two objects in a triangulated category one has the ``$\Ext$-space'' as follows.

\begin{defn}\label{def-ext-in-DA} Given any DG modules $M,N\in \mc{D}(A)$ and $i\in \Z$, we set the \emph{$i$th $p$-DG $\Ext$-space}, or simply the \emph{$\Ext$-space} between them, to be
\[
\Ext_{\mc{D}(A)}^i (M,N):=\Hom_{\mc{D}(A)}(M,N[i]),
\]
and we define \emph{the total $\Ext$-space} as
\[
\Ext^\bullet_{\mc{D}(A)}(M,N):=\bigoplus_{i\in \Z}\Ext^i_{\mc{D}(A)}(M,N).
\]
\end{defn}
One can also make these definitions in the homotopy category $\mc{C}(A)$ but we will not need them. It follows from \eqref{eqn-morphism-space-as-invariants} that there is a bi-functorial isomorphism
\begin{equation}\label{eqn-ext-space-as-invariants}
\Ext^\bullet_{\mc{D}(A)}(M,N)\cong \bigoplus_{i\in \Z}\mH^i(\RHOM_{A}(M,N)).
\end{equation}

We next recall the notion of compact modules, which takes place in the derived category.

\begin{defn} \label{def-compact-mod} Let $A$ be a DG algebra. A DG module $M$ over $A$ is called \emph{compact} (in the derived category $\mc{D}(A)$) if and only if, for any family of DG modules $N_i$ where $i$ takes value in some index set $I$, the natural map
\[
\oplus_{i\in I}\Hom_{\mc{D}(A)}(M,N_i)\lra \Hom_{\mc{D}(A)}(M,\oplus_{i\in I} N_i)
\]
is an isomorphism of $\Bbbk$-vector spaces.
\end{defn}

The strictly full subcategory of $\mc{D}(A)$ consisting of compact modules will be denoted by $\mc{D}^c(A)$.  It is triangulated and will be referred to as the \emph{compact derived category}.

For positive DG algebras, we have a close relationship between finite cell modules and compact modules. Denote by $\mc{F}(A)$ the collection of all finite cell modules over $A$, and let $\overline{\mc{F}}(A)$ be the closure of $\mc{F}(A)$ inside $\mc{D}(A)$ under isomorphisms. Notice that the natural inclusion functor $\overline{\mc{F}}(A)\subset \mc{D}(A)$ factors through the compact derived category $\mc{D}^c(A)$.

\begin{prop}\label{prop-characterization-compactness}The natural inclusion functor
$$\overline{\mc{F}}(A)\subset \mc{D}^c(A)$$
is an equivalence of triangulated categories. \hfill$\square$
\end{prop}

In other words, over a positive DG algebra, any compact module is quasi-isomorphic to a finite cell module. The proof of the Proposition can be found in \cite{SchPos}. Applying the usual Grothendieck group construction to the triangulated category $\mc{D}^c(A)$, we obtain the DG Grothendieck group of a DG algebra $A$, which we denote by $K_0(A)$. The abelian group stucture on $K_0(A)$ arises from the module structure $K_0(\Bbbk)\otimes K_0(A)\to K_0(A)$ induced by the categorical action of $\mc{D}(\Bbbk)$ on $\mc{D}^c(A)$.
%there is some identification with Z chosen, so maybe not natural? a Z/2=Z^\times torsor, donno

\paragraph{Resolutions.} Let $M$ be a DG module. By a \emph{DG resolution} of $M$ we mean a quasi-isomorphism of DG modules $P\lra M$ such that $P$ is cofibrant. It is known that resolutions for DG modules always exist \cite[Theorem 3.1]{Ke1}, although they are usually not unique. However, it is not true in general that any DG module has a resolution by a finite cell module. Proposition \ref{prop-characterization-compactness} says that, over a positive DG algebra, this happens precisely when the module is compact in the derived category.

\begin{ntn}\label{ntn-finite-cell-mod}In what follows, we introduce a convenient way to rewrite finite cell modules. Let $M$ be a DG module with a finite increasing filtration $F^\bullet$ ($\bullet\in \Z$), with $F^i/F^{i-1}:= P_i$. Then we will write this filtered module schematically as a direct sum of its associated graded, and put an arrow from $P_i$ to $P_j$ decorated by a map $\phi$ if $\dif_M(p_i)$ projects onto $\phi(p_i)\in P_j$ (necessarily $j < i$), where $p_i$ is a lift of an element of $P_i$ inside $M$. For instance, the following diagram stands for a filtered DG module
\[
\xymatrix{ P_n \ar[r]^-{\phi_n} & P_{n-1}\ar[r]^-{\phi_{n-1}} \ar[r] & \cdots\ar[r]^-{\phi_{m+2}} & P_{m+1} \ar[r]^-{\phi_{m+1}}& P_m }
\]
with the filtration implicitly understood as $F^i=\sum_{k=n}^iP_k$. On each $F^i$, the differential structure is defined inductively by the composition
\[
\dif_{F^i}: P_j \lra P_j+P_{j-1}\subset F^i, \ \ \ \ p_j\mapsto (\dif_{P_j}(p_j), \phi_j(p_j)),
\]
where $m\leq j \leq i$.
\end{ntn}

Alternatively, one can regard the original filtration on $M$ as obtained from collapsing the bigrading of a bigraded $A$-module $M\cong \bigoplus_{i,j\in \Z}P_j^i$ into a single grading, which will always be referred to as the \emph{$q$-grading}, and take $F^j:=\sum_{k\geq j, i\in \Z}P^i_k$. Note that $\dif_M$ does not necessarily preserve the bigrading.

For instance, in this notation, the filtered DG module in Example \ref{eg-finite-cofibrant-mod} would be represented in this notation by
\[
0\lra A \stackrel{\cdot a_0}{\lra} A \lra 0.
\]
%%%HERE%%%
%For another example, given any DG module $M$, its homological shift $M[1]$ will be depicted by the filtered module
%\[
%0\lra \underbrace{M = M = \cdots =M}_{p-1~\textrm{copies}} \lra 0.
%\]
%More generally, if $V_i$ denotes the $(i+1)$-dimensional indecomposable $p$-complex, then $M\o V_i$ will be written as
%\[
%0\lra \underbrace{M = M = \cdots =M}_{i+1~\textrm{copies}} \lra 0.
%\]

%Using this notation, we give a way of describing the (co)cone construction, in analogue with the usual DG case.

\begin{lemma}\label{lemma-cone-construction}
Let $\phi:M\lra N$ be a map of DG modules over $A$. The cone of the morphism $\phi$ is isomorphic to the filtered DG module
\[
\xymatrix{ M[1]\ar[r]^-{\phi} & N \ ,}
\]
where $N$ sits in homological degree zero. The cocone of $\phi$ is isomorphic to the filtered module
\[
\xymatrix{ M  \ar[r]^-{\phi} &  N[-1] \ ,}
\]
where $M$ sits in homological degree zero.
\end{lemma}

For later convenience, we will simply write the (co)cone of $\phi$ as
\[
M[1] \xrightarrow{\phi} N \quad \quad \left(\textrm{resp.}~M \xrightarrow{\phi}N[-1]\right).
\]

\begin{proof}
This follows from the definition of the cone and cocone. See \cite[Definition 2.14]{KQ}.
\end{proof}

\begin{lemma}\label{lemma-ses-gives-p-DG-resolution} \cite[Lemma 3.11]{QiSussan1} Let $0\lra K \stackrel{\phi}{\lra} L  \stackrel{\psi}{\lra} M \lra 0$ be a short exact sequence of graded modules over the smash product algebra $A_\dif$. Then the filtered DG module
\begin{subequations}
\begin{equation}\label{eqn-ses-lead-to-acyclic-pdg-mod-1}
 K[2] \stackrel{\phi}{\lra} L[1]  \stackrel{\psi}{\lra} M
\end{equation}
%\begin{equation}\label{eqn-ses-lead-to-acyclic-pdg-mod-2}
%0\lra \underbrace{K=\cdots =K}_{p-1} \stackrel{\phi}{\lra} L \stackrel{\psi}{\lra}  \underbrace{M=\cdots =M}_{p-1}\lra 0
%\end{equation}
\end{subequations}
is acyclic.
\end{lemma}

\paragraph{Smooth DG algebras.}
We now recall the notion of \emph{smooth} DG algebras. Throughout, for simplicity, we will assume that $A$ is a finite-dimensional DG algebra.

\begin{defn}\label{def-hopfological-finite-algebra}The DG algebra $A$ is called \emph{left (resp.~right) smooth} if all simple left (resp.~right) DG modules are compact in the derived category of left (resp.~right) DG modules.
\end{defn}

Given a finite-dimensional DG algebra $A$, we have a strictly full subcategory $\mc{D}^f(A)\subset \mc{D}(A)$ which consists of DG left (resp.~right) modules that are quasi-isomorphic to finite-dimensional left (resp.~right) DG-modules. The natural inclusion of $\mc{D}^c(A)$ into $\mc{D}(A)$ factors through $\mc{D}^f(A)$, since $A$ is finite-dimensional.

The following result is well-known. For completeness, we record a reference in which the result is proven in the context of hopfological algebra.

\begin{lemma}\label{lemma-equivalence-DC-DF} \cite[Lemma 3.15]{QiSussan1} If $A$ is left (resp.~right) smooth, then the natural inclusion functor $\mc{D}^c(A)\lra \mc{D}^f(A)$ is an equivalence of triangulated categories of left (resp.~right) DG modules.
\end{lemma}

\paragraph{Background on $A_\infty$-algebras.}
We quickly review some basic material on $A_\infty$-algebras over a field of characteristic two, following the references \cite{KeIntro} and \cite{KeCont}.

An $A_\infty$-algebra $A$ over a field $\Bbbk$ is a $\Z$-graded vector space over $\Bbbk$ with operations $m_n \colon A^{\otimes n} \rightarrow A$ of degree $2-n$ satisfying relations
\[
\sum m_u(1^{\otimes r} \otimes m_s \otimes 1^{\otimes t})=0 ,
\]
where the sum is over all decompositions $n=r+s+t$ and $u=r+1+t$.

A right $A_\infty$-module over $A$ is a $\Z$-graded vector space $M$ with operations $m_n \colon M \otimes A^{\otimes n-1} \rightarrow M$ of degree $2-n$ satisfying relations similar to those of an $A_\infty$-algebra.  For details see \cite[Section 4.2]{KeIntro}.  A left $A_\infty$-module is defined in a similar way.

A morphism of right $A_\infty$-modules over $A$ is sequence of maps 
$ f_n \colon M \otimes A^{\otimes n-1} \rightarrow N$ of degree $1-n$, such that for each $n \geq 1$, there is a relation
\[
\sum f_u \circ (1^{\otimes r} \otimes m_s \otimes 1^{\otimes t})
=
\sum m_u(f_r \otimes 1^{\otimes s}) ,
\]
where the left hand sum is over all decompositions $n = r+s+t, r, t \geq 0, s \geq 1$, 
$u = r + 1 + t$, and the right hand sum is over all decompositions
$n = r + s, r \geq 1, s \geq 0$, and $ u = 1 + s$.
The composition of morphisms is defined in \cite[Section 4.2]{KeIntro}.

An $A_{\infty}$-bimodule $M$ is a $\Z$-graded vector space with operations $m_{i,j} \colon A^{\otimes i-1} \otimes M \otimes A^{\otimes j-1} \rightarrow M $ of degree $3-i-j$ satisfying certain relations.  For more details see \cite[Section 3]{KeCont}.  One can form a tensor product of a right $A_{\infty}$-module $M$ and a left $A_\infty$-module $N$ over an $A_\infty$-algebra $A$, which we denote as $M \boxtimes_{A} N$.  Again see \cite[Section 3]{KeCont} for details.

%\[
%\begin{DGCpicture}[scale=0.5]
%\DGCPLstrand(0,0)(1,1)
%\DGCPLstrand(2,0)(1,1)
%\DGCPLstrand(4,0)(3,1)
%\DGCPLstrand(1,1)(2,2)
%\DGCPLstrand(3,1)(2,2)
%\DGCPLstrand(2,2)(2,3)
%\end{DGCpicture}
%~=~
%\begin{DGCpicture}[scale=0.5]
%\DGCPLstrand(0,0)(1,1)
%\DGCPLstrand(2,0)(3,1)
%\DGCPLstrand(4,0)(3,1)
%\DGCPLstrand(1,1)(2,2)
%\DGCPLstrand(3,1)(2,2)
%\DGCPLstrand(2,2)(2,3)
%\end{DGCpicture}
%\]

%\[
%\begin{DGCpicture}[scale=0.5]
%\DGCPLstrand(0,0)(1,-1)
%\DGCPLstrand(2,0)(1,-1)
%\DGCPLstrand(4,0)(3,-1)
%\DGCPLstrand(1,-1)(2,-2)
%\DGCPLstrand(3,-1)(2,-2)
%\DGCPLstrand(2,-2)(2,-3)
%\end{DGCpicture}
%\]

%\[
%\begin{DGCpicture}[scale=0.5]
%\DGCPLstrand(0,0)(1,-1)
%\DGCPLstrand(2,0)(3,-1)
%\DGCPLstrand(4,0)(3,-1)
%\DGCPLstrand(1,-1)(2,-2)
%\DGCPLstrand(3,-1)(2,-2)
%\DGCPLstrand(2,-2)(2,-3)
%\end{DGCpicture}
%\]

%\[
%\begin{DGCpicture}
%\DGCPLstrand(0,0)(0,2)
%\DGCdot*{1}[r]{$m_1$}
%\end{DGCpicture}
%\quad \quad
%\begin{DGCpicture}
%\DGCPLstrand(0,0)(1,1)
%\DGCPLstrand(2,0)(1,1)
%\DGCPLstrand(1,1)(1,2)
%\DGCdot*.{1}[r]{$m_2$}
%\end{DGCpicture}
%\]

%\[
%\begin{DGCpicture}
%\DGCPLstrand(0,0)(1,1)
%\DGCPLstrand(2,0)(1,1)
%\DGCPLstrand(1,1)(1,2)
%\DGCdot*.{1}[r]{$m_i$}
%\DGCcoupon*(0,0)(2,0.2){$\cdots$}
%\end{DGCpicture}
%\]

\section{The DG algebra \texorpdfstring{$A_n^{!}$}{A!}}
\label{sec:dualzigzagalgebra}
%\comment{QY: For this paper, I suggest that we take $d$ to be of degree 1. The arrows going right would then have degree 1, and arrows going left would have degree zero. This degree convention would match better with the usual homological algebra, $\{1\}=[1]$, and explain why the formality assumption of Seidel-Thomas fails in this case.}

\paragraph{The quiver presentation.} Let $\Bbbk$ be a field of characteristic $2$. Let $n$ be a natural number greater than or equal to two, and $Q_n$ be the following quiver:
\begin{equation}\label{quiver-Q}
\xymatrix{
 \overset{1}{\circ} &
 \cdots \ltwocell{'}&
 \overset{i-1}{~\circ~}\ltwocell{'}&
 \overset{i}{\circ}\ltwocell{'}&
 \overset{i+1}{~\circ~}\ltwocell{'}&
 \cdots \ltwocell{'}&
 \overset{n}{\circ}\ltwocell{'}
 }
\end{equation}
Let $\Bbbk Q_n$ be the path algebra associated to $Q_n$ over the ground field $\Bbbk$. Following Khovanov-Seidel \cite{KS}, we use, for instance, the symbol $(i|j|k)$, where $i,j,k$ are vertices of the quiver $Q_n$, to denote the path which starts at a vertex $i$, then goes through $j$ (necessarily $j=i\pm 1$) and ends at $k$.
The composition of paths is given defined by
\begin{equation}
(i_i|i_2|\cdots|i_r)\cdot (j_1|j_2|\cdots|j_s)=
\left\{
\begin{array}{ll}
(i_i|i_2|\cdots|i_r|j_2|\cdots|j_s) & \textrm{if $i_r=j_1$,}\\
0 & \textrm{otherwise,}
\end{array}
\right.
\end{equation}
where $i_1,\dots, i_r$ and $j_1, \dots, j_s$ are sequences of neighboring vertices in $Q_n$.

\begin{defn}\label{def-algebra-An}The algebra $A_n^!$ is the quotient of the path algebra $\Bbbk Q_n$ by the relations
\[
(i|i-1|i)=(i|i+1|i)~~ (i=2,\dots, n-1),~~ (1|2|1)=0.
\]
This is a graded algebra where the degree of the generators are given as follows. The degree of $(i|i+1)$ is 1, the degree of $(i|i-1)$ is 0, and the degree of $(i)$ is 0.
\end{defn}

\begin{rmk}
In the literature, the grading on $A_n^!$ is usually taken to be the path length grading.  

The algebra $A_n^!$ is Morita equivalent to a maximally singular block of category $\mathcal{O}(\mathfrak{gl}_n)$ and is isomorphic to a particular Webster algebra.  For more details see \cite{QiSussan1}.

The algebra $A_n^!$ is Koszul, whose quadratic dual is isomorphic to the algebra $A_n$ considered in \cite{KS}.
\end{rmk}

The quiver algebra $A_n^!$ has a differential $\dif$, given on the path generators by
\[
\dif_{}(i|i+1)=(i|i+1|i|i+1), \quad \quad \dif_{}(i|i-1)=0,
\]
and extended to the whole algebra via the Leibnitz rule. Sometimes for convenience, we will write the loop $(i|i+1|i)=(i|i-1|i)$ at the vertex $i$ simply as $c_i$, so that $\dif(i|i+1)= c_i(i|i+1)= (i|i+1)c_{i+1}$.

In the rest of the paper, we will be exclusively focusing on a family of DG algebras that are positive (Definition \ref{defn-positive-p-DGA}).  Now let $A$ be a positive DG algebra, and $\{e_i|i=1,\dots, n\}$ be a complete list of pairwise non-isomorphic indecomposable idempotents in $A$. Define $P_i$ (resp. $_iP$) to be the DG module $A\cdot e_i$ (resp. $e_i\cdot A$), and $L_i:=P_i/A^{\prime}\cdot P_i$ (resp. $_iL:={_iP}/{_iP\cdot A^{\prime}}$), where $A^{\prime}=\oplus_{i>0}A^i$ is the augmentation ideal of $A$. Then $A$ is left (resp. right) smooth if and only if all $L_i$'s (resp. $_iL$'s), $i=1,\dots,n$ are compact in the derived category. Furthermore, the natural categorical pairing, which exists for any finite-dimensional DG algebra,
\begin{equation}\label{eqn-categorical-pairing}
\RHOM_A(-,-) : \mc{D}^c(A) \times \mc{D}^f(A) \lra \mc{D}(\Bbbk)
\end{equation}
gives rise to a pairing on $\mc{D}^c(A)$ via the equivalence $\mc{D}^c(A)\lra \mc{D}^f(A)$ of Lemma \ref{lemma-equivalence-DC-DF}. It descends to a semi-linear non-degenerate pairing on the Grothendieck group
\begin{equation}\label{eqn-K-group-pairing}
[\RHOM_A(-,-)] : K_0(A) \times K_0(A) \lra \mathbb{Z},
\end{equation}
which is perfect if all $L_i$'s ($i=1,\dots, n$) are absolutely irreducible.

\begin{lemma} \label{lemma-DG-ny-resolution}
\begin{enumerate}
\item[(i)] For $i=1,\ldots,n-1$, the left DG module $L_i$ is quasi-isomorphic to the finite cell module $\mathbb{P}(L_i)$
\begin{equation}\label{pdgresLi1}
\begin{gathered}
\xymatrix{   & P_{i-1} \ar[dd]^{_{(i-1|i|i+1)}}\ar[dr]^-{(i-1|i)} & \\
  P_i [1] \ar[ur]^-{(i|i-1)}  \ar[dr]_-{(i|i+1)} && P_i \ ,\\
& P_{i+1} [1] \ar[ur]_-{(i+1|i)} & }
\end{gathered}
\end{equation}
The right DG module ${}_iL$ is quasi-isomorphic to the finite cell module $\mathbb{P}({}_iL)$
\begin{equation}\label{pdgresiL1}
\begin{gathered}
\xymatrix{   & {}_{i-1} P[1] \ar[dr]^-{(i|i-1)} & \\
  {}_i P [1] \ar[ur]^-{(i-1|i)}  \ar[dr]_-{(i+1|i)} &&{}_iP \ ,\\
& {}_{i+1} P\ar[uu]^{_{(i-1|i|i+1)}} \ar[ur]_-{(i|i+1)} & }
\end{gathered}
\end{equation}
\item[(ii)] The left DG module $L_n$ is quasi-isomorphic to the finite cell module $\mathbb{P}(L_n)$
\begin{equation}\label{pdgresLn}
\begin{gathered}
\xymatrix{ 
  & P_2  \ar[r]^{(2|3)}  \ar[dd]_{(2|1)} \ar[ddr] & \cdots \ar[r] & P_{n-2}  \ar[rr]^{(n-2|n-1)} \ar[dd]_{(n-2|n-3)} \ar[ddrr] && P_{n-1}  \ar[dr]^{(n-1|n)} \ar[dd]^{(n-1|n-2)}  &&  \\
  P_1  \ar[ur]^{(1|2)} \ar[dr] & & & & & &  P_n  \ ,\\
  & P_1 [-1 ] \ar[r]_{(1|2)} \ar[uur]| (.75) {(1|2|3)} & \cdots \ar[r] & P_{n-3} [ -1 ] \ar[rr]_{(n-3|n-2)} \ar[uurr]| (.65)  {(n-3|n-2|n-1)} & & P_{n-2} [-1] \ar[ur]_-{(n-2|n-1|n)} &&
}
\end{gathered}  
\end{equation}
where all the undecorated arrows are identity maps. The right DG module ${}_nL$ is quasi-isomorphic to the finite cell module $\mathbb{P}({}_nL)$
\begin{equation}\label{pdgresnL}
\begin{gathered}
\xymatrix{
 {}_{n-1}P[1]  \ar[rr]^{(n|n-1)}  && {}_n P  \ .
}
\end{gathered} 
\end{equation}
\end{enumerate}
\end{lemma}

\begin{proof}
The first statement is the $p=2$ case of \cite{QiSussan1}.

It is straightforward to check that the module in \eqref{pdgresLn} is in fact a DG module.
This object $\mathbb{P}(L_n)$ could be filtered by starting from $P_n$ and adding modules moving down and to the left and then up.  More precisely:
\[
F_1=P_n , \quad
F_2=P_n \oplus P_{n-2} [-1] , 
\ldots , 
F_{2j-1} = F_{2j-2} \oplus P_{n-j+1} , \quad
F_{2j} = F_{2j-1} \oplus P_{n-j-1}[-1] ,  \ldots,
F_{2n-2} = \mathbb{P}(L_n) \ .
\]
\end{proof}

\begin{cor}
    The DG algebra $(A_n^!,d)$ is smooth. 
\end{cor}
\begin{proof}
    The lemma shows that simple left and right DG modules over $A_n^!$ are finite-cell, and thus are compact objects in the derived category. 
\end{proof}

\paragraph{Temperley-Lieb action.} Now we recall a categorical Temperley-Lieb algebra action on $\mc{D}^c(A_n^!)$ constructed in \cite{QiSussan1}. 

\begin{defn}\label{defn-cup-cap-functors}  Let $i$ be an integer in $\{1,\dots, n-1\}$.
\begin{enumerate}
\item[(i)] The \emph{$i$th cup functor} is the functor
\[
\cup_i \colon \mathcal{D}(\Bbbk) \lra \mathcal{D}(A_n^!) , \quad \quad \cup_i(V) : = L_i \otimes V.
\]

\item[(ii)] The \emph{$i$th cap functor} is given by
\[
\cap_i \colon \mathcal{D}(A_n^!) \lra  \mathcal{D}(\Bbbk) , \quad \quad \cap_i(M) : = {}_i L \otimes^\mathbf{L}_{A_n^!} M.
\]
\end{enumerate}
\end{defn}

Thanks to Lemma \ref{lemma-DG-ny-resolution}, the above functors interacting between $\mc{D}(A_n^!)$ and $\mc{D}(\Bbbk)$ send compact objects to compact objects. Furthermore, there is a functor isomorphism
\begin{equation}\label{eqn-cap-RHOM}
    \cap_i(M)= {}_iL\otimes_{A_n^!}^{\mathbf{L}}(M) \cong \RHOM_{A_n^!}(L_i,M)
\end{equation}
In turn, this allows one to define, for each $i\in \{1,\dots, n-1\}$, an endo-functor
\begin{equation}\label{eqn-def-TL-generator}
\mf{U}^!_i = \cup_i \circ \cap_i : \mc{D}^c(A_n^!) \lra \mc{D}^c(A_n^!),
\end{equation}
which is usually diagrammatically presented as 
\[
\mf{U}^!_i=
\begin{DGCpicture}[scale=0.8]
\DGCstrand(-0.5,0)(-0.5,2)[$^1$`$\empty$]
\DGCstrand(0.25,0)(0.25,2)[$^2$`$\empty$]
\DGCstrand(1,0)(1,0.25)(2,0.25)(2,0)/d/[$^i$`$^{i+1}$]
\DGCstrand/d/(1,2)(1,1.75)(2,1.75)(2,2)
\DGCstrand(3.5,0)(3.5,2)[$^{n-1}$`$\empty$]
\DGCstrand(2.75,0)(2.75,2)[$^n$`$\empty$]
\DGCcoupon*(0.45,0.5)(0.95,1.5){$\dots$}
\DGCcoupon*(2.15,0.5)(2.65,1.5){$\dots$}
\end{DGCpicture}
\ .
\]
This diagrammatic notation explains the names ``cups'' and ``caps'' introduced earlier.

\begin{thm}\cite[Theorem 5.10]{QiSussan1}
\label{thm-TL-action}
For $ i=1, \ldots, n-1 $, the functors $ \mf{U}^!_i $ are self-biadjoint, and they enjoy the following functor-isomorphism relations.
\begin{enumerate}
\item[(i)] $ \mf{U}^!_i \circ \mf{U}^!_i \cong \mf{U}^!_i  \oplus \mf{U}^!_i[1]  $,
\item[(ii)] $ \mf{U}^!_i \circ \mf{U}^!_j \cong \mf{U}^!_j \circ \mf{U}^!_i $ for $ |i-j|>1 $,
\item[(iii)] $ \mf{U}^!_i \circ \mf{U}^!_j \circ \mf{U}^!_i \cong \mf{U}^!_i $ for $ |i-j|=1 $.
\end{enumerate}
\end{thm}

Via the self-adjointness of the functors $\mf{U}^!_i$, one can define, for each $i \in \{1,\dots, n-1\}$, the \emph{braid functors} on $\mc{D}^c(A_n^!)$ as follows. For any $M\in \mc{D}(A_n^!)$, set
\begin{equation}
\mf{T}^!_i(M) : =L_i \o \RHOM_{A_n^!}(L_i,M)[1] \xrightarrow{\lambda} M.
\end{equation}
Here $\lambda$ is the canonical pairing map. Alternatively, $\lambda$ can be identified with the counit map of the biadjoint pair $\cup_i$ and $\cap_i$ (identifying $\cap_i(M)$ with $\RHOM_{A_n^!}(L_i,M)$ via Equation \eqref{eqn-cap-RHOM}).

\begin{thm}\cite[Theorem 5.14]{QiSussan1}
\label{thm-braidrelations}
The functors $ \mf{T}^!_i$, $ i=1, \ldots, n-1 $, are invertible on $\mc{D}^c(A_n^!)$. Furthermore, they satisfy the following categorical braid relations:
\begin{enumerate}
\item $ \mf{T}^!_i  \mf{T}^!_j \cong \mf{T}^!_j \mf{T}^!_i $ if $ |i-j|>1 $,
\item $ \mf{T}^!_i \mf{T}^!_{i+1} \mf{T}^!_i \cong \mf{T}^!_{i+1} \mf{T}^!_{i} \mf{T}^!_{i+1} $.
\end{enumerate}
\end{thm}

\paragraph{Non-trivial maps between simples.} \label{nontrivmaps}
Lemma \ref{lemma-DG-ny-resolution} allows one to construct explicit DG module maps between the cofibrant replacements $\mathbb{P}(L_i)$. We record them here for later use.

The unique non-trivial map (up to homotopy) $ \alpha_{i,i+1} \colon L_{i+1}  \rightarrow L_i  $ is given by:
\begin{equation}\label{map-alpha-i-i+1}
\begin{gathered}
\xymatrix@C=4em{
 & P_{i}  
 \ar[dd]^{} \ar[dr]^{} 
 \ar@/^2.0pc/@[red][rdddd]|(.63){(i)} &  \\
P_{i+1}[1] \ar[ur]^{} \ar[dr]^{} 
\ar@/_2.0pc/@[red][ddddr]|(.3){(i+1)}  
& &  P_{i+1} 
\\
& P_{i+2} [1] \ar[ur]^{} 
%\ar[rdd]|(.3){(i+1) }  
& & \\
%%%%
 & P_{i-1}  \ar[dd]^{} \ar[dr]_(.5){} &  \\
P_{i} [1]\ar[ur]_{} \ar[dr]^{} & &  P_{i}  \\
& P_{i+1} [1]  \ar[ur]_{}   & & \\
}
\end{gathered} \ .
\end{equation}

The unique (up to homotopy) map $ \alpha_{i+1,i} \colon L_i \rightarrow L_{i+1} $ is given by:
\begin{equation}\label{map-alpha-i+1-i}
\begin{gathered}
\xymatrix@C=4em{
 & P_{i-1}  \ar[dd]^{} \ar[dr]^{} 
 \ar@/^2.5pc/@[red][ddd]|(.8){(i-1|i)}
 \ar@/_2.0pc/@[red][ddddl]|(.7){(i-1|i|i+1)}
 &  \\
P_i [1]  \ar[ur]^{} \ar[dr]^{}  
\ar@/_2.0pc/@[red][ddd]|(.7){(i|i+1)}
\ar@/_2.0pc/@[red][ddr]|(.7){(i)}
& &  P_i   \\
& P_{i+1} [1]  \ar[ur]^{}  
 \ar@/^2.5pc/@[red][ddr]|(.8){(i+1)}
& & \\
%%%
 & P_{i}  \ar[dd]_{} \ar[dr]_(.5){} &  \\
P_{i+1}[1] \ar[ur]_{} \ar[dr]^{} & &  P_{i+1}
\\
& P_{i+2} [1]\ar[ur]_{}   & & \\
}
\end{gathered} \ .
\end{equation}

%\begin{equation}
%\xymatrix@C=4em{
% & P_{i-1} \{ -1 \} \ar[dd]^{(i-1|i|i+1)} \ar[dr]^{(i-1|i)} &  \\
%P_i\{-2 \} \ar[ur]^{(i|i-1)} \ar[dr]^{(i|i+1)} \ar[ddr]_{(i)} & &  P_i \ar[ddd]^{(i|i+1)} \ar@/^2.0pc/@[red][ldddd]|(.43){(i|i+1|i+2) }\\
%& P_{i+1} \{-1 \} \ar[ur]^{(i+1|i)} \ar[rdd]|(.3){(i+1) }  \ar@/^2.0pc/@[red][ddd]|(.7){(i+1|i+2)}& & \\
% & P_{i} \{ -2 \} \ar[dd]_{(i|i+1|i+2)} \ar[dr]_(.5){(i|i+1)} &  \\
%P_{i+1}\{-3 \} \ar[ur]_{(i+1|i)} \ar[dr]^{(i+1|i+2)} & &  P_{i+1} \{ -1 \} \\
%& P_{i+2} \{-2 \} \ar[ur]_{(i+2|i+1)}   & & \\
%}
%\end{equation}

The unique (up to homotopy) non-trivial map $ \alpha_{n-1,n} \colon L_{n}  \rightarrow L_{n-1} $ is given as follows:

\begin{equation}
\begin{gathered}
\xymatrix{ 
P_1 \ar[r] \ar[dr] & P_2 \ar[r]  \ar[d] \ar[dr] & \cdots \ar[r] & P_{n-2} \ar[r] \ar[d] \ar[dr] 
& P_{n-1} \ar[r] \ar[d]  
\ar@/_3.5pc/@[red][ddd]|(.43){(n-1)}
\ar@/^2.0pc/@[red][ddddl]|(.53){(n-1|n)}
& P_n \\
& P_1 [-1]\ar[r] \ar[ur] & \cdots \ar[r] & P_{n-3} [-1]\ar[r] \ar[ur] & P_{n-2} [-1]\ar[ur] 
\ar@/_2.0pc/@[red][dd]|(.33){(n-2|n-1)}
\ar@/^5.0pc/@[red][dddl]|(.53){(n-2|n-1|n)}
\\
%%%%
&& & P_{n-2} \ar[dd]  \ar[dr] 
\\
&& P_{n-1} [1] \ar[ur] \ar[dr] & & P_{n-1} \\
&& & P_{n} [1] \ar[ur]
}
\end{gathered} \ .
\end{equation}

The unique (up to homotopy) non-trivial map $ \alpha_{n,n-1} \colon L_{n-1}\rightarrow L_{n}  $ is given by:

\begin{equation}
\begin{gathered}
\xymatrix{ 
&& & P_{n-2} \ar[dd]  \ar[dr] 
\ar@/^2.0pc/@[red][ddddr]|(.63){(n-2)}
\\
&& P_{n-1} [1]\ar[ur] \ar[dr] 
\ar@/_2.0pc/@[red][ddrr]|(.43){(n-1)}
& & P_{n-1} \\
&& & P_{n} [1] \ar[ur] 
\ar@/^2.0pc/@[red][drr]|(.63){(n)}
\\
%%%%
P_1 \ar[r] \ar[dr] & P_2 \ar[r]  \ar[d] \ar[dr] & \cdots \ar[r] & P_{n-2} \ar[r] \ar[d] \ar[dr] 
& P_{n-1} \ar[r] \ar[d]  
& P_n \\
& P_1 [-1] \ar[r] \ar[ur] & \cdots \ar[r] & P_{n-3} [-1] \ar[r] \ar[ur] & P_{n-2} [-1] \ar[ur] 
}
\end{gathered} \ .
\end{equation}

\paragraph{Some homotopies.}
Define the map $h_i \colon L_i \rightarrow L_i$
for $i=1,\ldots,n-1$ by
\begin{equation}
\begin{gathered}
\xymatrix@C=4em{
 & P_{i-1}  \ar[dd]_{} \ar[dr]_(.5){} 
\ar@/^2.0pc/@[red][ddd]|(.43){(i-1)}
 &  \\
P_{i}[1] \ar[ur]_{} \ar[dr]^{} 
\ar@/^2.0pc/@[red][ddd]|(.43){(i)}
& &  P_{i}  \\
& P_{i+1} [1] \ar[ur]_{}   & & \\
%%%
 & P_{i-1}  \ar[dd]_{} \ar[dr]_(.5){} &  \\
P_{i}[1] \ar[ur]_{} \ar[dr]^{} & &  P_{i}  \\
& P_{i+1} [1] \ar[ur]_{}   & &
} 
\end{gathered} \ .
\end{equation}

Let $ h_n \colon L_n \rightarrow L_n$ be the map
\begin{equation}
\begin{gathered}
\xymatrix{ 
P_1 \ar[r] \ar[dr] & P_2 \ar[r]  \ar[d] \ar[dr] & \cdots \ar[r] & P_{n-2} \ar[r] \ar[d] \ar[dr] 
& P_{n-1} \ar[r] \ar[d]  
& P_n 
\ar@/^2.0pc/@[red][dd]|(.43){(n)}
 \\
& P_1 [-1] \ar[r] \ar[ur] & \cdots \ar[r] & P_{n-3} [-1] \ar[r] \ar[ur] & P_{n-2} [-1] \ar[ur] 
\\
%%%%
P_1 \ar[r] \ar[dr] & P_2 \ar[r]  \ar[d] \ar[dr] & \cdots \ar[r] & P_{n-2} \ar[r] \ar[d] \ar[dr] 
& P_{n-1} \ar[r] \ar[d]  
& P_n \\
& P_1 [-1] \ar[r] \ar[ur] & \cdots \ar[r] & P_{n-3} [-1] \ar[r] \ar[ur] & P_{n-2} [-1] \ar[ur] 
}
\end{gathered} \ .
\end{equation}

Finally, we define maps $h_{n,i} \colon L_i \rightarrow L_n$
\begin{equation}
\begin{gathered}
\xymatrix{ 
&&  & P_{i-1}  \ar[dd]_{} \ar[dr]_(.5){} 
\ar@/^2.0pc/@[red][dddd]|(.63){(i-1)}
&  &&\\
&& P_{i}[1] \ar[ur]_{} \ar[dr]^{} 
\ar@/_2.0pc/@[red][rdd]|(.43){(i)}
& &  P_{i}  &&\\
&&  & P_{i+1} [1]\ar[ur]_{}   & & &&\\
%%%%
P_1 \ar[r] \ar[dr] & P_2 \ar[r]  \ar[d] \ar[dr] & \cdots \ar[r]
& P_i \ar[r] \ar[d] \ar[dr]
%& P_{i+1} \ar[r] \ar[d]
& \cdots \ar[r]
& P_{n-2} \ar[r] \ar[d] \ar[dr] 
& P_{n-1} \ar[r] \ar[d]  
& P_n \\
& P_1 [-1]\ar[r] \ar[ur] & \cdots \ar[r] 
& P_{i-1} [-1]\ar[r] \ar[ur]
%& P_{i} \ar[r]
& \cdots \ar[r]
& P_{n-3} [-1]\ar[r] \ar[ur] & P_{n-2} [-1]\ar[ur] 
}
\end{gathered} \ .
\end{equation}

\paragraph{Auxiliary maps.}
It is notationally convenient to define
\[
L_0 :=
\begin{gathered}
\xymatrix{   & 0 \ar[dd]^{}\ar[dr]^-{} & &\\
 0 \ar[ur]^-{}  \ar[dr]_-{} && 0 \\
& P_{1} [1]\ar[ur]_-{} && }
\end{gathered} \ .
\]
Now we define maps
$\alpha_{1,0} \colon L_0 \rightarrow L_1$ 
\begin{equation}
\xymatrix@C=4em{
 & 0 \ar[dd]^{} \ar[dr]^{}  &  \\
0 \ar[ur]^{} \ar[dr]^{}   & &  0 \\
& P_{1} [1] \ar[ur]^{} 
\ar@/^2.0pc/@[red][ddr]|(.63){(1)}
& & \\
%%%%
 & 0  \ar[dd]_{} \ar[dr]_(.5){} &  \\
P_{1}[1] \ar[ur]_{} \ar[dr]^{} & &  P_{1}  \\
& P_{2} [1] \ar[ur]_{}   & & \\
}
\end{equation}

and 
$\alpha_{0,1} \colon L_1 \rightarrow L_0$ 
\begin{equation}
\begin{gathered}
\xymatrix@C=4em{
 & 0  \ar[dd]_{} \ar[dr]_(.5){} &  \\
P_{1}[1] \ar[ur]_{} \ar[dr]^{} 
\ar@/_2.0pc/@[red][ddddr]|(.33){(1)}
& &  P_{1}  \\
& P_{2} [1] \ar[ur]_{}   & & \\
%%%
 & 0 \ar[dd]^{} \ar[dr]^{}  &  \\
0 \ar[ur]^{} \ar[dr]^{}   & &  0 \\
& P_{1}[1] \ar[ur]^{}  & & \\
}
\end{gathered} \ .
\end{equation}

Note that $\alpha_{1,0} \alpha_{0,1}$ is null-homotopic.

\section{The \texorpdfstring{$A_{\infty}$}{A-infinity} zigzag algebra}\label{sec:ainfzz}

\paragraph{The zigzag algebra.}
We now recall the definition of the algebra $A_n$ that served a crucial role in \cite{KS}. When graded by path length, this algebra is Koszul dual to the algebra $A_n^!$ when it is also graded by path length.

\begin{defn}\label{def-algebra-zigzag} The ring $A_n$ is the quotient of the path algebra $\Bbbk Q_n$ by the relations
\[
(i|i \pm 1|i \pm 2)=0,
~~ (i|i+1|i)=(i|i-1|i), \text{ for }
~~ (i=2,\dots, n-1),~~ (n|n-1|n)=0 ,
\]
\[
\deg((i))=0 , \quad \deg((i|i+1))=0 , \quad \deg((i+1|i))=1 .
\]

The zigzag algebra $C_{n-1}$ is idempotented subalgebra given by
\[
C_{n-1} = ((1)+\cdots+(n-1)) A_n ((1)+\cdots+(n-1)) .
\]
\end{defn}

For later use, we define the projective modules
\begin{equation}\label{eqn-Zi}
   Z_i:= A_n\cdot (i), \quad i=1,\dots, n,    
\end{equation}
which are the indecomposable direct summands of $A_n$. Then $C_{n-1}$ can also be identified with
\begin{equation}
    C_{n-1}\cong\END_{A_n}(\oplus_{i=1}^{n-1} Z_i) \ .
\end{equation}
\paragraph{The dual DG algebra $E_n$.}
We now define an intermediate DG algebra $E_n$ as the graded endomorphism algebra 
\begin{subequations}
\begin{equation}\label{eqn-En}
    E_n:= \END_{A_n^!}(\oplus_{i=1}^n \mathbb{P}(L_i)).
\end{equation}
The differential on the endormphism algebra, as usual, is given by
\begin{equation}\label{eqn-En-d}
    d(f)(x):=d(f(x))+f(d(x))
\end{equation}
    \end{subequations}
for any homogeneous $f\in E_n$ and $x\in \oplus_{i=1}^n \mathbb{P}(L_i) $. Here the usual sign that should appear in \eqref{eqn-En-d} is dropped since we are working in characteristic two. 

In $E_n$, we have the natural idempotents $1_i$, $i=1,\dots, n$, which act as identity maps on $\mathbb{P}(L_i)$ and zero on $\mathbb{P}(L_j)$ for $j\neq i$. We will also use an idempotent truncation of $E_n$ as follows. Consider $e=1_1+\cdots+1_{n-1}$.
Then we have the idempotented subalgebra
$S_{n-1} := e E_n e$.

We will use the cofibrant DG modules
\begin{equation}\label{eqn-Qi}
    Q_i:= E_n 1_i, \quad i=1,\dots, n,
\end{equation}
which, by construction, is a direct DG summand of $E_n$.

For the computation of the $A_\infty$-structure on the homology of $E_n$, we will consider a subalgebra of $E_n$ defined by the explicit DG module maps recorded above. Let $E_n^\prime$ be the $\Bbbk$-algebra generated by
\begin{itemize}
    \item $1_i$ for $i=1,\ldots, n$, $\deg(1_i)=0$.
    \item $\alpha_{i+1,i}, \alpha_{i,i+1}$ for $i=1,\ldots, n-1$, $\deg(\alpha_{i+1,i})=1, \deg(\alpha_{i,i+1})=0$.
    \item $h_i$ for $i=1,\ldots,n$, $\deg(h_i)=0$.
    \item $h_{n,i}$ for $i=1,\ldots,n-2$, $\deg(h_{n,i})=1$.
\end{itemize}
We set $h_{n,n-1}=\alpha_{n,n-1}$.

\begin{prop}
    These generators satisfy the following set of relations:
    \begin{itemize}
\item The elements $1_i$ are orthogonal idempotents.
\item The idempotents $1_i$  act by identity on the other generators.  
\item $\alpha_{i,i+1} \alpha_{i+1,i+2}=0$, for $i=1,\ldots, n-2$.
\item $\alpha_{i+2,i+1} \alpha_{i+1,i}=0$, for $i=1,\ldots, n-3$.
\item $h_i^2=h_i$, for $i=1,\ldots,n$.
\item $h_i \alpha_{i,i+1}=0 $, for $i=1,\ldots,n-1$.
\item $ \alpha_{i,i+1} h_{i+1} = \alpha_{i,i+1} $, for $i=1,\ldots,n-2$.
\item $ \alpha_{i,i+1} \alpha_{i+1,i} \alpha_{i,i+1}=0$.
\item $ \alpha_{n-1,n} h_{n} =  0 $.
\item $\alpha_{i+1,i} h_i = h_{i+1} \alpha_{i+1,i} $, for $i=1,\ldots,n-2$.
\item $\alpha_{n,n-1} h_{n-1} + h_n \alpha_{n,n-1}=\alpha_{n,n-1} $.
\item $h_{n,i} h_i=h_{n,i}$, for $i=1,\ldots,n-2$.
\item $h_n h_{n,i}=0$, for $i=1,\ldots,n-2$.
\item $h_{n,i} \alpha_{i,i+1} =0$, for $i=1,\ldots,n-2$.
\end{itemize}

The $DG$ structure is defined on the generators by
\begin{itemize}
\item $\dif(1_i)=0$,
\item $\dif(\alpha_{i+1,i})=\dif(\alpha_{i,i+1})=0$,
\item $\dif(h_i)= \alpha_{i,i-1}\alpha_{i-1,i}+ \alpha_{i,i+1}\alpha_{i+1,i}$, for $i=1,\ldots,n-1$,
\item $\dif(h_{n,i})=h_{n,i+1} \alpha_{i+1,i} $.
\end{itemize}
\end{prop}

\begin{proof}
These are straightforward computations. 
\end{proof}

One could easily compute bases of the idempotented subspaces
$1_i E_n^\prime 1_j $ and their homologies
$1_i \mH(E_n^\prime) 1_j $ as follows.

First assume $i,j \neq n$.
\[
1_i E_n^\prime 1_i=\Bbbk \langle 1_i, h_i, \alpha_{i,i+1} \alpha_{i+1,i}, \alpha_{i,i-1} \alpha_{i-1,i}  \rangle
\quad \quad d(h_i)=\alpha_{i,i+1} \alpha_{i+1,i} + \alpha_{i,i-1} \alpha_{i-1,i}
\]
\[
\mH(1_i E_n^\prime 1_i)=\Bbbk \langle 1_i, c_i \rangle
\]

\[
1_i E_n^\prime 1_j = 0 \quad \text{for} \quad |i-j|>1
\]

\[
1_i E_n^\prime 1_{i+1} = \Bbbk \langle \alpha_{i,i+1} \rangle
\]
\[
\mH(1_i E_n^\prime 1_{i+1}) = \Bbbk \langle \alpha_{i,i+1} \rangle
\]

\[
1_{i+1} E_n^\prime 1_{i} = \Bbbk \langle \alpha_{i+1,i}, \alpha_{i+1,i} \alpha_{i,i+1} \alpha_{i+1,i}, \alpha_{i+1,i} h_i \rangle \quad \quad d(\alpha_{i+1,i} h_i)=\alpha_{i+1,i} \alpha_{i,i+1} \alpha_{i+1,i}
\]
\[
\mH(1_{i+1} E_n^\prime 1_{i}) = \Bbbk \langle \alpha_{i+1,i} \rangle
\]

Second, assume $i$ or $j$ is $n$.
\[
1_n E_n^\prime 1_n = \Bbbk \langle 1_n , h_n, \alpha_{n,n-1} \alpha_{n-1,n}  \rangle \quad \quad d(h_n)=\alpha_{n,n-1} \alpha_{n-1,n}
\]
\[
\mH(1_n E_n^\prime 1_n) = \Bbbk \langle 1_n \rangle
\]
\[
1_n E_n^\prime 1_i = \Bbbk \langle h_{n,i}, h_{n,i+1} \alpha_{i+1,i}  \rangle \quad \quad d(h_{n,i})=h_{n,i+1} \alpha_{i+1,i}   \quad \text{ for } i=1,\ldots,n-2 
\]
\[
\mH(1_n E_n^\prime 1_i)=0 \quad \text {for } i=1,\ldots,n-2
\]

\[
1_i E_n^\prime 1_n = 0 \text{ for } i=1,\ldots,n-2 
\]
\[
\mH(1_i E_n^\prime 1_n)=0 \quad \text {for } i=1,\ldots,n-2
\]
\[
1_n E_n^\prime 1_{n-1} = \Bbbk \langle \alpha_{n,n-1}, \alpha_{n,n-1} \alpha_{n-1,n} \alpha_{n,n-1}, h_n \alpha_{n,n-1} \rangle \quad \quad d(h_n \alpha_{n,n-1})=\alpha_{n,n-1} \alpha_{n-1,n} \alpha_{n,n-1}
\]
\[
\mH(1_n E_n^\prime 1_{n-1}) = \Bbbk \langle \alpha_{n,n-1} \rangle
\]

\[
1_{n-1} E_n^\prime 1_n = \Bbbk \langle \alpha_{n-1,n} \rangle
\]
\[
\mH(1_{n-1} E_n^\prime 1_n) = \Bbbk \langle \alpha_{n-1,n} \rangle
\]

\begin{prop}
   There is an isomorphism of algebras $A_n \cong \mH(E_n^\prime) \cong \mH(E_n)$. 
\end{prop}

\begin{proof}
    This is clear from the relations in the algebras and the computations of the homology above.
\end{proof}

\paragraph{$A_{\infty}$-structure on $C_{n-1}$.}
Consider the idempotented subalgebra $C_{n-1} \subset \mH(E_{n})$.  We compute the $A_{\infty}$-structure on
the idempotented algebra $C_{n-1} \cong \mH(S_{n-1})$
using the perturbation lemma.

%\[
%\begin{gathered}
%\xymatrix{   
%e (A_{n+1}^!)^! e \ar@(ul,ur)^H \ar@/^1.0pc/@[][rrr]^{p} 
%&&& C_n \ar@/^1.0pc/@[][lll]^{j}
%}
%\end{gathered}
%\]

Let $\overrightarrow{\alpha}_i = \alpha_{i,i+1} \alpha_{i+1,i}$ and $\overleftarrow{\alpha}_i = \alpha_{i,i-1} \alpha_{i-1,i}$.
Now we define linear maps $p$, $j$, and $H$ as follows.

\begin{align*}
p(h_i) &= 0 \\
p(\alpha_{i,i+1}) &=(i|i+1) \\
p(\alpha_{i+1,i}) &=(i+1|i) \\
p(\overrightarrow{\alpha}_i) &=(i|i+1|i) \\
p(\overleftarrow{\alpha}_i) &=(i|i+1|i) \\
\end{align*}

\begin{align*}
j((i|i+1)) &= \alpha_{i,i+1} \\
j((i+1|i)) &= \alpha_{i+1,i} \\
j((i|i+1|i)) &= \overrightarrow{\alpha}_i \\
\end{align*}

\begin{align*}
H(\overrightarrow{\alpha}_i) &= 0 \\
H(\overleftarrow{\alpha}_i) &= h_i \\
H(\alpha_{i+1,i} \overrightarrow{\alpha}_i) &= \alpha_{i+1,i} h_i \\
H(\overrightarrow{h}_i) &= 0 \\
H(\alpha_{i+1,i} h_i) &= 0 .
\end{align*}

This data is summarized in \eqref{eq:pertdata}.
\begin{equation} \label{eq:pertdata}
\begin{gathered}
\xymatrix@!0{ 
S_{n-1} \ar@(dl,ul)^{H} \ar@/^1pc/[rrr]^p
&&& C_{n-1} \ar@/^1pc/[lll]^j 
}
\end{gathered}
\end{equation}

The perturbation lemma (see \cite[Theorem 2.3]{KeCont}) determines an
$A_{\infty}$-structure on $C_{n-1}$ is given by 
\begin{equation} \label{eq:tree}
m_k = \sum_T m_k^T
\end{equation}
where $T$ ranges over planar rooted binary trees with $k$ leaves and $m_k^T$ is given by composing the tree-shaped diagram obtained by labeling each leaf by $j$, each branch point by multiplication $\mu$ in $S_{n-1}$, each internal edge by $H$, and the root by $p$.
Note that this $A_\infty$-structure is unique up to $A_{\infty}$-isomorphism.

\begin{thm}\label{thm:ainfformulas}
All higher operations $m_k$ vanish for $k>3$.  There are three families of $m_3$ operations:
\begin{align*}
m_3((i|i+1), (i+1|i), (i|i+1|i)) &=  (i|i+1|i) \\
m_3((i-1|i), (i|i-1), (i-1|i)) &=  (i-1|i) \\
m_3((i|i+1|i), (i|i-1), (i-1|i)) &=  (i|i+1|i) .
\end{align*}
The non-zero trees giving rise to these operations are:

\[
\begin{DGCpicture}[scale=0.5]
\DGCPLstrand(0,0)(1,-1)
%\DGCdot*.{-.5}[r]{$j$}
%\DGCdot*.{-.1}[d]{$(i|i+1)$}
\DGCPLstrand(2,0)(3,-1)
%\DGCdot*.{-1}[d]{$j$}
%\DGCdot*.{1}[l]{$(i+1|i)$}
\DGCPLstrand(4,0)(3,-1)
%\DGCdot*.{0}[r]{$j$}
%\DGCdot*.{1}[l]{$(i|i+1|i)$}
\DGCPLstrand(1,-1)(2,-2)
\DGCPLstrand(3,-1)(2,-2)
\DGCPLstrand(2,-2)(2,-3)
%\DGCdot*{1}[r]{$m_1$}
%\DGCdot*.{1}[r]{$m_2$}
\DGCcoupon*(-.5,0)(.5,-1){$j$}
\DGCcoupon*(1.5,0)(2.5,-1){$j$}
\DGCcoupon*(3.5,0)(4.5,-1){$j$}
\DGCcoupon*(2.5,-1.2)(3.5,-2.2){$H$}
\DGCcoupon*(.7,-2)(2.7,-3){$p$}
\DGCcoupon*(-3,0)(0,.5){$(i|i+1)$}
\DGCcoupon*(.5,0)(3.25,.5){$(i+1|i)$}
\DGCcoupon*(3.5,0)(6.5,.5){$(i|i+1|i)$}
\end{DGCpicture}
\ ,
\quad
%%%%
\begin{DGCpicture}[scale=0.5]
\DGCPLstrand(0,0)(1,-1)
%\DGCdot*.{-.5}[r]{$j$}
%\DGCdot*.{-.1}[d]{$(i|i+1)$}
\DGCPLstrand(2,0)(3,-1)
%\DGCdot*.{-1}[d]{$j$}
%\DGCdot*.{1}[l]{$(i+1|i)$}
\DGCPLstrand(4,0)(3,-1)
%\DGCdot*.{0}[r]{$j$}
%\DGCdot*.{1}[l]{$(i|i+1|i)$}
\DGCPLstrand(1,-1)(2,-2)
\DGCPLstrand(3,-1)(2,-2)
\DGCPLstrand(2,-2)(2,-3)
%\DGCdot*{1}[r]{$m_1$}
%\DGCdot*.{1}[r]{$m_2$}
\DGCcoupon*(-.5,0)(.5,-1){$j$}
\DGCcoupon*(1.5,0)(2.5,-1){$j$}
\DGCcoupon*(3.5,0)(4.5,-1){$j$}
\DGCcoupon*(2.5,-1.2)(3.5,-2.2){$H$}
\DGCcoupon*(.7,-2)(2.7,-3){$p$}
\DGCcoupon*(-3,0)(0,.5){$(i-1|i)$}
\DGCcoupon*(.5,0)(3.25,.5){$(i|i-1)$}
\DGCcoupon*(3.5,0)(6.5,.5){$(i-1|i)$}
\end{DGCpicture}
\ ,
\quad
%%%%
\begin{DGCpicture}[scale=0.5]
\DGCPLstrand(0,0)(1,-1)
%\DGCdot*.{-.5}[r]{$j$}
%\DGCdot*.{-.1}[d]{$(i|i+1)$}
\DGCPLstrand(2,0)(3,-1)
%\DGCdot*.{-1}[d]{$j$}
%\DGCdot*.{1}[l]{$(i+1|i)$}
\DGCPLstrand(4,0)(3,-1)
%\DGCdot*.{0}[r]{$j$}
%\DGCdot*.{1}[l]{$(i|i+1|i)$}
\DGCPLstrand(1,-1)(2,-2)
\DGCPLstrand(3,-1)(2,-2)
\DGCPLstrand(2,-2)(2,-3)
%\DGCdot*{1}[r]{$m_1$}
%\DGCdot*.{1}[r]{$m_2$}
\DGCcoupon*(-.5,0)(.5,-1){$j$}
\DGCcoupon*(1.5,0)(2.5,-1){$j$}
\DGCcoupon*(3.5,0)(4.5,-1){$j$}
\DGCcoupon*(2.5,-1.2)(3.5,-2.2){$H$}
\DGCcoupon*(.7,-2)(2.7,-3){$p$}
\DGCcoupon*(-3,0)(0,.5){$(i|i+1|i)$}
\DGCcoupon*(.5,0)(3.25,.5){$(i|i-1)$}
\DGCcoupon*(3.5,0)(6.5,.5){$(i-1|i)$}
\end{DGCpicture}
\ .
\]
\end{thm}

\begin{proof}
It is clear that these trees give rise to non-trivial operations stated in the theorem and that the other terms $m_k^T$ in \eqref{eq:tree} are zero in these cases.

In order to see that trees with more than 3 inputs are zero, first note that the possible outputs of $H$ are $h_i$ and $\alpha_{i+1,i} h_i $.  

If there are only $3$ leaves in a tree, then it is straightforward to see that the only non-trivial possibilities are the ones stated in the theorem.

If there are more than three leaves, then, at the root of the tree, two outputs of $H$ must be multiplied in $E_{n}^\prime$ non-trivially and then mapped to something non-zero in $C_n$ via $p$.  Since the only non-zero outputs of $H(x)H(y)$ are $h_i$ and $\alpha_{i+1,i} h_i$, and these are mapped to zero by $p$, we conclude that there could not be any non-trivial operations $m_k$ with $k>3$.
\end{proof}

\paragraph{Non-vanishing of Hochschild cohomology.}
%\subsection{Non-vanishing in Hochschild cohomology}\label{sec:class}
In this section, we prove that the map $m_3 \in C^{3}(C_{n-1},C_{n-1})$ is not a Hochschild coboundary when $n\geq 4$. In more detail, the transfer theorem gave the assignments in Theorem \ref{thm:ainfformulas}. After extending by zero, the map $m_3 : C_{n-1}^{\otimes 3} \to C_{n-1}$ can be interpreted as a degree $-1$ element $m_3 \in C^3(C_{n-1},C_{n-1})^{-1}$ 
of the Hochschild cochain complex,
$$\cdots \to C^2(C_{n-1},C_{n-1})^{-1} \xrightarrow{\delta} C^3(C_{n-1},C_{n-1})^{-1} \to \cdots,$$
computing the degree $-1$ part of Hochschild cohomology of the algebra $C_{n-1}$. We show here that there is no map $m\in C^2(C_{n-1},C_{n-1})^{-1}$ which satisfies $\delta m = m_3$. Since the boundary map $\delta$ preserves degree and the transfer theorem guarantees $\delta m_3 = 0$, the computation below suffices to show that $m_3$ determines a non-trivial class in $\mHH^3(C_{n-1},C_{n-1})^{-1}$.

Now, for any map $m\in C^2(C_{n-1},C_{n-1})^{-1}$, i.e., $m$ is an element in $\Hom(C_{n-1}^{\otimes 2}, C_{n-1}[1])$, homogeneity implies that there are constants  $\alpha_i$, $b_i$, $c_i$, $d_i$ and  $\eta_i$ indexed by the idempotents $i$ of $C_{n-1}$, $1\leq i \leq n-1$, such that
\begin{align*}
    m((i|i+1),(i+1|i)) &= \alpha_i(i)\\
    m((i|i+1|i),(i|i+1)) &= b_i(i|i+1)\\
    m((i|i+1|i), (i|i-1)) &= c_i(i|i-1)\\
    m((i|i+1|i), (i|i+1|i)) &= d_i(i|i+1|i)\\
    m((i|i-1),(i-1|i|i-1)) &= \eta_i(i|i-1).\\
\end{align*}
By evaluating, the equation
$$(\delta m)(x,y,z) = xm(y,z) - m(xy,z) + m(x,yz) - m(x,y)z$$
on inputs $(x,y,z)\in C^{\otimes 3}$ of the form 
$((i|i+1),(i+1|i),(i|i-1))$, $((i|i-1),(i-1|i|i-1),(i-1|i))$, $((i|i+1),(i+1|i),(i|i+1|i))$ and $((i|i+1|i),(i|i+1),(i+1|i))$,
we obtain the equations
\begin{align*}
c_i + \alpha_i &=0\\
b_{i-1}+\eta_i &=0 \\
\eta_{i+1} + d_i + \alpha_i &= 1\\
\alpha_i + d_i + b_i &= 0
\end{align*}
which are inconsistent. To produce these equations, $(i-1),(i),(i+1)$ must be in $C_{n-1}$, so $n\geq 4$.

\section{Categorical symmetries} \label{sec:braidaction}
In this section, we construct a (weak) categorical Temperley-Lieb algebra and braid group action on the compact derived category of the $A_\infty$-algebra $C_{n-1}$.
For notational simplicity, we will often let $C=C_{n-1}$.

For $i=1,\ldots,n-1$, let $C_i:=C(i)$ and ${}_iC:=(i)C$. Also set $B_i=C_i \otimes {}_iC$. This $(C,C)$-bimodule can be regarded as arising from the cohomology of the DG bimodule $S_{n-1} 1_i\otimes 1_i S_{n-1}$ over $S_{n-1}\otimes S_{n-1}$. It follows that $B_i$ has a unique minimal $A_\infty$-structure over $C\otimes C$.

\begin{lemma}
The bimodule $B_i$ has the
$A_{\infty}$-structure 
\[
m_{k,l} \colon C^{\otimes k-1} \otimes B_i \otimes C^{\otimes l-1} \rightarrow B_i
\]
of degree $3-k-l$ by
\[
m_{k,l}= 0 \text{ if } k \text{ and } l > 0 \quad \quad
m_{k,1}= m_k \otimes \Id \quad \quad
m_{1,l} = \Id \otimes m_l
\]
\end{lemma}

\begin{proof}
It is straightforward to check that this defines an $A_{\infty}$-structure on the bimodule.    
\end{proof}

\begin{lemma}
On the $(C,C)$-bimodule $C$, there is an $A_{\infty}$-structure 
\[
m_{i,j} \colon C^{\otimes i-1} \otimes C \otimes C^{\otimes j-1} \rightarrow C
\]
of degree $3-i-j$ by $m_{i,j}=m_{i+j-1}$.
\end{lemma}

\begin{proof}
It is also straightforward to check that this is an $A_{\infty}$-structure on the bimodule $C$.  For more details see \cite[Lemma 4.1]{Tradler}.
\end{proof}

\begin{lemma}
There is an $A_{\infty}$-bimodule map $f_{i,j} \colon B_k \rightarrow C$ defined by
$f_{i,j}=0$ unless $i=2, j=1$ or $i=1, j=2$.  In these other cases
\begin{align*}
f_{2,1} (\alpha \otimes [\beta(k) \otimes (k) \gamma)]) &= m_3(\alpha, \beta(k), (k) \gamma) \\
f_{1,2} ([\alpha(k) \otimes (k) \beta)] \otimes \gamma) &= m_3(\alpha(k), (k) \beta, \gamma) .
\end{align*}
\end{lemma}

\begin{proof}
It is a straightforward, yet tedious, calculation that this defines an $A_{\infty}$-map.
\end{proof}

\begin{defn} Consider the derived category $\mc{D}_\infty(C)$, and let $k\in \{1,\dots, n-1\}$.
    \begin{itemize}
        \item[(i)] The functor $\mc{U}_k$ is the derived tensor functor
        \[
        \mf{U}_k: \mc{D}_\infty(C) \lra \mc{D}_\infty(C), \quad M\mapsto B_k\boxtimes_{C}M.
        \]
        \item[(ii)] The functor $\mc{T}_k$ is the derived tensor functor with the one-sided twisted complex
        \[
        \mf{T}_k: \mc{D}_\infty(C) \lra \mc{D}_\infty(C), \quad M\mapsto (B_k \stackrel{f}{\lra} C)\boxtimes_{C}M.
        \]
    \end{itemize}
\end{defn}

\begin{thm} Let $\mc{D}^c_\infty(C)$ be the compact derived category of the $A_\infty$-algebra $C$. 
    \begin{enumerate}
        \item[(1)] The functors $\mc{U}_i$, $i\in \{1,\dots, n-1\}$ define a  categorical Temperley-Lieb algebra action on $\mc{D}^c_\infty(C)$.
        \item[(2)] The functors $\mc{T}_i$, $i\in \{1,\dots, n-1\}$ define a  categorical braid group action on $\mc{D}^c_\infty(C)$.
    \end{enumerate}
\end{thm}
\begin{proof}
    The proof essentially reduces to the categorical braid group action of $\mathrm{Br}_n$ on $\mc{D}^c(A_n^!)$ constructed in \cite{QiSussan1} specialized over a field of characteristic two.
    
    By construction, there is an equivalence of DG derived categories
    \begin{subequations}

    \begin{equation} \label{eqn-der-equ-1}
    \mc{D}(A_n^!) \cong \mc{D}(E_n) .      \end{equation}
    On the other hand, the algebra $A_n$ is the cohomology of $E_n$ with its minimal $A_\infty$-algebra structure. By standard arguments about $A_\infty$-algebras (see, for instance, \cite[Section 3.3]{KeAinfty}), there is a quasi-isomorphism of $A_\infty$-algebras between $E_n$ and $A_n$, which induces a triangulated equivalence
    \begin{equation} \label{eqn-der-equ-2}\mc{D}(E_n)\cong \mc{D}_\infty(A_n).
    \end{equation}        
    \end{subequations}
     The derived equivalences \eqref{eqn-der-equ-1} and \eqref{eqn-der-equ-2} descend to equivalences of the corresponding compact derived categories.
    
    Next, the $A_\infty$-bimodule structures on the bimodules defining $\mf{U}_i$ and $\mf{T}_i$ can also be naturally identified with the $A_\infty$-bimodule structure inherited from the functors $\mc{U}_i^!$ and $\mf{T}^!_i$ on $\mc{D}(A_n^!)$.

    In \cite{QiSussan1}, the Temperley-Lieb algebra $\mathrm{TL}_n$ and the $n$-stranded braid group $\mathrm{Br}_n$ act on $\mc{D}(A_n^!)$, generated by $\mf{U}_i^!$ and $\mf{T}_i^!$, $i=1,\dots, n-1$. By Lemma \ref{lemma-DG-ny-resolution}, one can  see that the categorical Temperley-Lieb generators  preserve the smallest triangulated subcategory generated by the first $n-1$ simple DG modules $L_1,\dots, L_{n-1}$:
    \begin{equation}
        \mc{U}_i^! (L_j) = \cup_i\circ \cap_i (L_j) =L_i\otimes (\mathbb{P}(L_i)\otimes_{A_n^!} L_j) =
        \left\{
        \begin{array}{cc}
            L_i \oplus L_i[1] & i=j  \\
            L_i & j=i-1 \\
            L_i[1] & j=i+1 \\
            0 & |i-j|>1
        \end{array}
        \right.
    \end{equation}
    Denote this subcategory by $\langle L_1,\dots, L_{n-1}\rangle$. Under the functor $\HOM_{A_n^!}(\oplus_{i=1}^n \mathbb{P}(L_i),\mbox{-})$, $\langle L_1,\dots, L_{n-1}\rangle$ corresponds to the triangulated subcategory generated by $Q_1,\dots, Q_{n-1}$ (equation \eqref{eqn-Qi}).
    This subcategory, by construction, is equivalent to the compact derived category over the idempotent truncated DG algebra $S_{n-1}$. 
    Upon taking cohomology, $Q_i$ has defines the $A_\infty$-modules, $Z_i$, whose $A_\infty$-endomorphism algebra is equal to $C_{n-1}$, which is the cohomology algebra of the truncated DG algebra $S_{n-1}$. We exhibit the containment relation of the categories involved as follows
    \begin{equation}\label{eqn-equivalences}
        \begin{gathered}
            \xymatrix{
            \mc{D}(A_n^!) \ar[rr]^{\HOM(\oplus_{i=1}^n\mathbb{P}(L_i), \mbox{-})} && \mc{D}(E_n) && \mc{D}_\infty(A_n) \ar[ll]_{\cong} \\
            \langle L_1,\dots, L_{n-1}\rangle \ar@{^{(}->}[u] && \langle Q_1,\dots, Q_{n-1}\rangle \ar@{^{(}->}[u] && \langle Z_1, \dots, Z_{n-1}\rangle \ar@{^{(}->}[u] \\
            && \mc{D}^c(S_{n-1})\ar[u]^{\cong}&&\mc{D}_\infty^c(C_{n-1})\ar[u]^{\cong}
            }
        \end{gathered} \ .
    \end{equation}
    Translating the corresponding action on $\mc{D}^c(A_n^!)$ to the $A_\infty$-algebra $C$ via the above diagram, the Temperley-Lieb algebra and braid group actions become actions on the derived category $\mc{D}_\infty(C_{n-1})$, preserving the full subcategory of compact objects.
\end{proof}

\begin{rmk}
    The objects $L_i$ ($i=1,\dots, n-1$) are analogues of an  \emph{$A_{n-1}$ chain of spherical objects} in the derived category $\mc{D}(A_n^!)$ in the sense of Seidel-Thomas  \cite{SeidelThomas}. Their images in $ \mc{D}_\infty(A_n) $ under the chain of equivalences
    \[ \mc{D}(A_n^!) \cong \mc{D}(E_n) \cong \mc{D}_\infty(A_n),\] 
    are given by the projective modules $C_i$. Thus the collection $C_i$, $i=1,\dots, n-1$
    also constitues such an $A_{n-1}$-chain of spherical objects. The braid group action thus naturally arises from spherical twists along this chain. We will point out the symplecto-geometric meaning of this collection of spherical objects in the next section.
\end{rmk}

\begin{cor}
The categorical braid group action descends to the Burau representation at $q=-1$ in the Grothendieck group.  More specifically, the following squares commute
\begin{equation*}
\xymatrix{
K_0(\mc{D}^c_\infty(C)) \ar[r]^-{\cong} \ar[d]^{[\mf{U}_i]} & V_{n-1} \ar[d]^{u_i} \\
K_0(\mc{D}^c_\infty(C)) \ar[r]^-{\cong} & V_{n-1}
}
\hspace{1in}
\xymatrix{
K_0(\mc{D}^c_\infty(C)) \ar[r]^-{\cong } \ar[d]^{[\mf{T}_i]} & V_{n-1} \ar[d]^{t_i} \\
K_0(\mc{D}^c_\infty(C)) \ar[r]^-{\cong } & V_{n-1}.
}
\end{equation*}
\end{cor}

\begin{proof}
This is a direct translation of \cite[Proposition 5.16]{QiSussan1} via the equivalences in the diagram \eqref{eqn-equivalences}.
\end{proof}

\section{Remarks on connections to symplectic topology} \label{sec:symp}

The zigzag algebras $C_{n-1}$ appear geometrically as endomorphism algebras of Lagrangians $\oplus_i L_i$ in Fukaya categories $\mathcal{D}_\infty(\mathcal{F}(S))$ \cite{KS}. When the grading conventions are those of Section \ref{sec:ainfzz}, the relevant symplectic manifold $S$ is a surface and the identification
\begin{equation}\label{eq:seidel}
C_{n-1}\cong \END_{\mathcal{D}_\infty (\mathcal{F}(S))}(\oplus_i L_i)
\end{equation}
endows the zigzag algebra with some auxiliary geometric content. The purpose of this section is to suggest how to locate the $A_\infty$-structure of Theorem \ref{thm:ainfformulas}, obtained by algebraic means, within the geometric construction of the Fukaya category.

Slightly oversimplifying, the Fukaya category $\mathcal{F}(S)$ of a surface $S$ has objects oriented embedded  1-manifolds $L\subset S$ with spin structure. When $L, L'\in Ob(\mathcal{F}(S))$ are transverse, morphisms $CF(L,L') := \mathbb{F}_2\langle L\pitchfork L' \rangle$ are spanned by points of intersection. The $A_\infty$-structure is encoded by maps $m_r : \mathcal{F}(S)^{\otimes r} \to \mathcal{F}(S)$ for $r\geq 1$.  If the collection $\{L_0,\ldots, L_r\}$ is pairwise transverse, then these maps
\begin{align}
m_r : CF(L_r,L_{r-1})\otimes \cdots \otimes CF(L_1,L_0)\to CF(L_{r},L_0) \quad\quad\textnormal{ are given by }\quad\quad \nonumber\\
m_r(x_r, \ldots, x_1) := \sum_{x_0\in L_n\pitchfork L_1} \#\mathcal{M}(x_{r},\ldots,x_0) x_0 \label{eq:mn}
\end{align}
 a (signed) count of immersed disks representing points in moduli spaces $\mathcal{M}(x_{r},\ldots,x_0)$ of virtual dimension zero \cite[II (13b)]{Seidel}. Figure \ref{fig:diskdiagram} depicts such a disk among Lagrangians which have been perturbed to be transverse. Each map $m_r$ can be extended to a map $m_r : \oplus_n \mathcal{F}(S)^{\otimes n}\to \mathcal{F}(S)$ so that the $A_{\infty}$-relations become $d^2=0$ after setting $d:=\sum_n m_n$. Since $\mathcal{F}(S)$ is too small to be triangulated, focus is placed on the triangulated hull $\mathcal{D}_\infty (\mathcal{F}(S))$ \cite[I (3h)]{Seidel}.

 Within this setting, consider simple closed curves arranged in the pattern depicted by the figure below.
\begin{center}
    \begin{overpic}[scale=1.5]{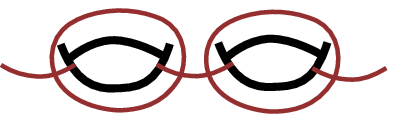}
\put(25.5,-5){$L_{i-1}$}
\put(67.5,-5){$L_{i+1}$}
\put(47.5,-5){$L_i$}
\put(5.5,-5){$\cdots$}
\put(87.5,-5){$\cdots$}
\end{overpic}
\end{center}
Each curve being a circle determines a spherical object in the Fukaya category. The braid group representation is generated by the twist functors 
$$T_{L_i} : \mathcal{D}_\infty (\mathcal{F}(S)) \to \mathcal{D}_\infty(\mathcal{F}(S))\quad\quad\quad\quad  T_{L_i}(X) := \mathrm{Cone}(\Hom(L_i,X)\otimes Y \xrightarrow{ev} X).$$
When $L_i$ has non-trivial spin structure, the action of $T_{L_i}$ agrees with the Dehn twist on $\mathcal{F}(S)$ \cite[Corollary 17.18]{Seidel}. For two curves, $L_i$ and $L_j$ above, the associated Dehn twists either commute or satisfy the braid relation depending upon whether the geometric intersection number is $0$ or $1$ respectively \cite[page 84]{FarbMargalit}. These objects determine the zigzag algebra $C_{n-1}$ in Equation \eqref{eq:seidel} above. The algebra $A_n$ can obtained by adjoining an embedded interval to $\oplus_i L_i$ \cite[Figure 3]{KS}.

When the curves $\{L_0,\ldots,L_r\}$ are not pairwise transverse, some care is needed to define the $A_\infty$-structure. One approach is to limit over all chain complexes associated to (isomorphic) Hamiltonian perturbations of Lagrangians, in this manner the map $m_r$ can be computed for any non-transverse collection $\{L_0,\ldots,L_r\}$ by any transverse family obtained by Hamiltonian isotopies. 
With this in mind, to compute $m_3((i|i+1), (i+1|i), (i|i+1))$ in Theorem \ref{thm:ainfformulas} we perturb both $L_i$ and $L_{i+1}$ to obtain isomorphic objects $\tilde{L}_i$  and $\tilde{L}_{i+1}$. Under the identifications made in the figure above, maps, such as $(i|i+1)$ correspond to distinct points of intersection, this results in the figure below.
\begin{figure}[htp]
    \centering
\begin{overpic}[scale=.85]{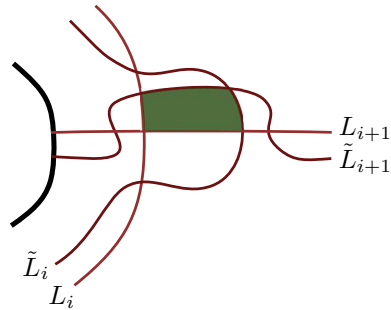}
 \put(102, 47.5){$L_{i+1}$}
 \put(102, 37.5){$\tilde{L}_{i+1}$}
 \put(12,-5){$L_i$}
\put(4,4){$\tilde{L}_i$}
\end{overpic}
    \caption{A disk among perturbations}\label{fig:diskdiagram}
\end{figure}
Here we suggest that, since we are working with coefficients in characteristic $2$, we can define a new map $\tilde{m}_3$ which counts the entire moduli space associated to the green disk. For a related comment see \cite[Remark 13.4]{Seidel}. Setting $\tilde{m}_r = 0$ for $r\neq 3$ and $\delta = \sum_r \tilde{m}_r$ the algebraic portion of this paper shows that $(d+\delta)^2=0$ and so an $A_\infty$-structure is obtained. We expect that $\delta$ admits a geometric definition as a sum over moduli spaces parallel to Equation \eqref{eq:mn}.

\bibliographystyle{alpha}
\bibliography{master}

\begin{thebibliography}{KQS17}

\bibitem[Bir74]{birman}
J.~S. Birman.
\newblock {\em Braids, links, and mapping class groups}.
\newblock Annals of Mathematics Studies, No. 82. Princeton University Press,
  Princeton, N.J.; University of Tokyo Press, Tokyo, 1974.

\bibitem[BK90]{BondalKapranov}
A.~I. Bondal and M.~M. Kapranov.
\newblock Enhanced triangulated categories.
\newblock {\em Mat. Sb.}, 181(5):669--683, 1990.
\newblock English version: \emph{Mathematics of the USSR-Sbornik}, 1991,
  70(1):93--107.

\bibitem[BW22]{BarWang}
S.~Barmeier and Z.~Wang.
\newblock ${A}_\infty$ deformations of extended {K}hovanov arc algebras and
  {S}troppel's conjecture.
\newblock 2022.
\newblock \href{http://arxiv.org/abs/2211.03354}{arXiv:2211.03354}.

\bibitem[CF94]{CF}
L.~Crane and I.~Frenkel.
\newblock Four dimensional topological quantum field theory, {H}opf categories,
  and the canonical bases.
\newblock {\em J. Math. Phys.}, 35(10):5136--5154, 1994.
\newblock \href{http://arxiv.org/abs/hep-th/9405183}{arXiv:hep-th/9405183}.

\bibitem[CK14]{ChenKho}
Y.~Chen and M.~Khovanov.
\newblock An invariant of tangle cobordisms via subquotients of arc rings.
\newblock {\em Fund. Math.}, 225(1):23--44, 2014.
\newblock \href{https://arxiv.org/abs/math/0610054}{arXiv:math/0610054}.

\bibitem[CS19]{CoSussan}
B.~Cooper and J.~Sussan.
\newblock Preprojective analogue of the cone construction.
\newblock {\em Homology Homotopy Appl.}, 21(2):171--198, 2019.

\bibitem[EQ16]{EQ1}
B.~Elias and Y.~Qi.
\newblock A categorification of some small quantum groups {II}.
\newblock {\em Adv. Math.}, 288:81--151, 2016.
\newblock \href{http://arxiv.org/abs/1302.5478}{arXiv:1302.5478}.

\bibitem[FKS05]{FKO}
I.~Frenkel, M.~Khovanov, and O.~Schiffmann.
\newblock Homological realization of {N}akajima varieties and {W}eyl group
  actions.
\newblock {\em Compos. Math.}, 141(6):1479--1503, 2005.
\newblock \href{https://arxiv.org/abs/math/0311485}{ arXiv:math/0311485}.

\bibitem[FM12]{FarbMargalit}
B.~Farb and D.~Margalit.
\newblock {\em A primer on mapping class groups}, volume~49 of {\em Princeton
  Mathematical Series}.
\newblock Princeton University Press, Princeton, NJ, 2012.

\bibitem[Kel94]{Ke1}
B.~Keller.
\newblock Deriving {DG} categories.
\newblock {\em Ann. Sci. \'Ecole Norm. Sup. (4)}, 27(1):63--102, 1994.

\bibitem[Kel01]{KeIntro}
B.~Keller.
\newblock Introduction to {$A$}-infinity algebras and modules.
\newblock {\em Homology Homotopy Appl.}, 3(1):1--35, 2001.
\newblock
  \href{http://front.math.ucdavis.edu/math.RA/9910179}{arXiv:math/9910179}.

\bibitem[Kel02]{KeAinfty}
B.~Keller.
\newblock {$A$}-infinity algebras in representation theory.
\newblock In {\em Representations of algebra. {V}ol. {I}, {II}}, pages 74--86.
  Beijing Norm. Univ. Press, Beijing, 2002.

\bibitem[Kel03]{KeLef}
B.~Keller.
\newblock Koszul duality and coderived categories (after {K}. {L}ef\`{e}vre).
\newblock 2003.
\newblock Available at
  \url{https://webusers.imj-prg.fr/~bernhard.keller/publ/kdc.pdf}.

\bibitem[Kel06a]{KeCont}
B.~Keller.
\newblock {$A$}-infinity algebras, modules and functor categories.
\newblock In {\em Trends in representation theory of algebras and related
  topics}, volume 406 of {\em Contemp. Math.}, pages 67--93. Amer. Math. Soc.,
  Providence, RI, 2006.
\newblock \href{https://arxiv.org/abs/math/0510508}{arXiv:math/0510508}.

\bibitem[Kel06b]{KeICM}
B.~Keller.
\newblock On differential graded categories.
\newblock {\em International Congress of Mathematicians}, II:151--190, 2006.
\newblock \href{http://arxiv.org/abs/math/0601185}{arXiv:math/0601185v5}.

\bibitem[Kho16]{Hopforoots}
M.~Khovanov.
\newblock {H}opfological algebra and categorification at a root of unity: the
  first steps.
\newblock {\em J. Knot Theory Ramifications}, 25(3):1640006, 26, 2016.
\newblock
  \href{http://front.math.ucdavis.edu/math.QA/0509083}{arXiv:math/0509083}.

\bibitem[KQ15]{KQ}
M.~Khovanov and Y.~Qi.
\newblock An approach to categorification of some small quantum groups.
\newblock {\em Quantum Topol.}, 6(2):185--311, 2015.
\newblock \href{http://arxiv.org/abs/1208.0616}{arXiv:1208.0616}.

\bibitem[KQS17]{KQS}
M.~Khovanov, Y.~Qi, and J.~Sussan.
\newblock p-{DG} cyclotomic nil{H}ecke algebras.
\newblock {\em to appear in Mem. Amer. Math. Soc.}, 2017.
\newblock \href{https://arxiv.org/abs/1711.07159}{arXiv:1711.07159}.

\bibitem[KS02]{KS}
M.~Khovanov and P.~Seidel.
\newblock Quivers, {F}loer cohomology, and braid group actions.
\newblock {\em J. Amer. Math. Soc.}, 15:203--271, 2002.
\newblock \href{http://arxiv.org/abs/math/0006056}{arXiv:math/0006056}.

\bibitem[LW23]{LiuWang}
J.~Liu and Z.~Wang.
\newblock ${A}_\infty$-deformations of zigzag algebras via {G}inzburg algebras.
\newblock 2023.
\newblock \href{http://arxiv.org/abs/2301.06757}{arXiv:2301.06757}.

\bibitem[Qi14]{QYHopf}
Y.~Qi.
\newblock Hopfological algebra.
\newblock {\em Compos. Math.}, 150(1):1--45, 2014.
\newblock \href{http://arxiv.org/abs/1205.1814}{arXiv:1205.1814}.

\bibitem[QS16]{QiSussan1}
Y.~Qi and J.~Sussan.
\newblock A categorification of the {B}urau representation at prime roots of
  unity.
\newblock {\em Selecta Math. (N.S.)}, 22(3):1157--1193, 2016.
\newblock \href{http://arxiv.org/abs/1312.7692}{arXiv:1312.7692}.

\bibitem[QS17]{QiSussan2}
Y.~Qi and J.~Sussan.
\newblock Categorification at prime roots of unity and hopfological finiteness.
\newblock In {\em Categorification and higher representation theory}, volume
  683 of {\em Contemp. Math.}, pages 261--286. Amer. Math. Soc., Providence,
  RI, 2017.
\newblock \href{https://arxiv.org/abs/1509.00438}{arXiv:1509.00438}.

\bibitem[QS18]{QiSussan3}
Y.~Qi and J.~Sussan.
\newblock $p$-{DG} cyclotomic nil{H}ecke algebras {II}.
\newblock {\em to appear in Mem. Amer. Math. Soc.}, 2018.
\newblock \href{https://arxiv.org/abs/1811.04372}{arXiv:1811.04372}.

\bibitem[RZ03]{RouZim}
R.~Rouquier and A.~Zimmermann.
\newblock Picard groups for derived module categories.
\newblock {\em Proc. London Math. Soc. (3)}, 87(1):197--225, 2003.

\bibitem[Sch11]{SchPos}
O.~M. Schn\"{u}rer.
\newblock Perfect derived categories of positively graded {DG} algebras.
\newblock {\em Applied Categorical Structures}, 19(5):757--782, 2011.
\newblock \href{http://arxiv.org/abs/0809.4782}{arXiv:0809.4782}.

\bibitem[Sei08]{Seidel}
P.~Seidel.
\newblock {\em Fukaya categories and {P}icard-{L}efschetz theory}.
\newblock Zurich Lectures in Advanced Mathematics. European Mathematical
  Society (EMS), Z\"{u}rich, 2008.

\bibitem[ST01]{SeidelThomas}
P.~Seidel and R.~Thomas.
\newblock Braid group actions on derived categories of coherent sheaves.
\newblock {\em Duke Math. J.}, 108(1):37--108, 2001.
\newblock \href{http://arxiv.org/abs/math/0001043}{arXiv:math/0001043}.

\bibitem[Str09]{StroppelGrass}
C.~Stroppel.
\newblock Parabolic category {O}, perverse sheaves on {G}rassmannians,
  {S}pringer fibres and {K}hovanov cohomology.
\newblock {\em Compositio Math.}, 145:954--992, 2009.
\newblock \href{http://arxiv.org/abs/math/0608234}{arXiv:math/0608234}.

\bibitem[Tra08]{Tradler}
T.~Tradler.
\newblock Infinity inner products on {A}-infinity algebras.
\newblock {\em J. Homotopy Relat. Struct.}, 3:245--271, 2008.
\newblock \href{http://arxiv.org/abs/math/0806.0065}{arXiv:math/0806.0065}.

\end{thebibliography}

%
% ====================================================================

\vspace{0.1in}

%\author{You Qi}
%\address{Department of Mathematics, University of Virginia,
%  Charlottesville, VA 22904, USA}
%\email{\href{mailto:yq2dw@virginia.edu}{yq2dw@virginia.edu}}
% \author{Joshua Sussan}
% \address{Department of Mathematics, CUNY Medgar Evers, Brooklyn, NY,
%   11225, USA}
%  \email{\href{mailto:jsussan@mec.cuny.ed}{jsussan@mec.cuny.edu}}

\vspace{0.1in}
\noindent B.~C.: {\sl \small Department of Mathematics, University of Iowa, Iowa City, IA, 52242, USA}
\newline\noindent  {\tt \small email: ben-cooper@uiowa.edu}

\noindent Y.~Q.: {\sl \small Department of Mathematics, University of Virginia, Charlottesville, VA, 22904, USA}
\newline\noindent  {\tt \small email: yq2dw@virginia.edu}

\noindent J.~S.: {\sl \small Department of Mathematics, CUNY Medgar Evers, Brooklyn, NY, 11225, USA}
\newline\noindent  {\tt \small email: jsussan@mec.cuny.edu}
\newline \noindent {\sl \small Mathematics Program\\
 The Graduate Center, CUNY, New York, NY 10016, USA}

% ==============================================================================
%
\end{document}